%% file: mw.tex
\documentclass[11pt]{article}
\usepackage[letterpaper, margin=1in]{geometry}
\usepackage{enumitem}
\usepackage{mathtools}
\usepackage{amsthm}
\usepackage{amssymb}
\usepackage{xspace}
\usepackage{bm}
\usepackage[breakable]{tcolorbox}
\usepackage{longtable}
\usepackage{booktabs}
\usepackage{graphicx}
\usepackage{relsize}
\usepackage[left, pagewise]{lineno}
\usepackage{tikz}
  \usetikzlibrary{positioning,fit,calc}
  \tikzstyle{block} = [draw=black, ultra thin, text width=2cm, minimum height=1.3cm, font = {\footnotesize\itshape},align=center]
  \tikzstyle{arrow} = [thick,->,>=stealth]

  \renewenvironment{quote}
  {\list{}{\leftmargin=0.5cm \rightmargin=0.5cm} 
   \item\relax}
  {\endlist}

  \usepackage[page,toc,titletoc,title]{appendix}
\usepackage[textsize=footnotesize]{todonotes}
\usetikzlibrary{shapes,arrows,positioning,patterns}
\usepackage{comment}
\usepackage{hyperref}
\usepackage{cleveref}
\hypersetup{
    colorlinks=true,
    citecolor=blue,
    linkcolor=blue,
    filecolor=blue,
    urlcolor=blue,
}
\usepackage{thm-restate}
\usepackage{placeins}
\usepackage[font=small,labelfont=bf,margin=0.06in]{caption}

\usepackage[style=alphabetic,maxnames=999]{biblatex}
\AtEveryBibitem{
  \clearfield{editor}
  \clearlist{editor}
  \clearname{editor}
  \clearfield{location}
  \clearlist{location}
  \clearname{location}
  \clearfield{isbn}
  \clearlist{isbn}
  \clearname{isbn}
  \clearfield{url}
  \clearlist{url}
  \clearname{url}
}

\RequirePackage[osf,sc]{mathpazo}
\usepackage{framed}
\usepackage{mdframed}
\usepackage{float}
\usepackage{xcolor}

\newcommand{\ind}[2][]{%
  \mathrel{
    \mathop{
      \vcenter{
        \hbox{\oalign{\noalign{\kern-.3ex}\hfil$\vert$\hfil\cr
              \noalign{\kern-.7ex}
              $\smile$\cr\noalign{\kern-.3ex}}}
      }
    }^{#2}\displaylimits_{#1}
  }
}

\newcommand{\eqdef}{\coloneqq}
\newcommand{\distFO}{\textrm{\upshape dist-FO}\xspace}

\DeclareMathOperator{\VCdim}{VCdim}

\DeclareMathOperator{\fw}{fw}

\DeclareMathOperator{\wcol}{wcol}

\DeclareMathOperator{\wreach}{WReach}

\newlength{\leftbarwidth}
\newlength{\leftbarsep}
\renewenvironment{leftbar}[1][blue]
{%
\def\FrameCommand
{%
{\hspace{-8pt} \color{#1} \vrule width 3pt}%
\hspace{0pt}
\fboxsep=\FrameSep\colorbox{#1!5!}%
}%
\MakeFramed{\hsize\hsize\advance\hsize-\width\FrameRestore}%
}
{\endMakeFramed}
\makeatletter

  \if@todonotes@disabled 
  \excludecomment{szin}
  \else 
  
  \fi

\makeatother

\newcommand\myitem[1]{\hyperref[item:#1]{\emph{\ref*{item:#1}.}{}}\xspace}

\newtheorem{theorem}{Theorem}[section]
\newtheorem*{theorem*}{Theorem}
\newtheorem{conjecture}[theorem]{Conjecture}

\newtheorem{corollary}[theorem]{Corollary}
\newtheorem*{corollary*}{Corollary}
\newtheorem{lemma}[theorem]{Lemma}
\newtheorem*{lemma*}{Lemma}
\newtheorem{fact}[theorem]{Fact}
\newtheorem{observation}[theorem]{Observation}

\newtheorem*{proposition*}{Proposition}

\newtheorem{claim}{Claim}[section]
\crefname{claim}{claim}{Claims} 
\Crefname{claim}{Claim}{Claims} 

\newenvironment{claimproof}[1][\proofname]{%
  \begin{proof}[#1]%
}{%
  \end{proof}%
}

\theoremstyle{remark}
\newtheorem{remark}[theorem]{Remark}
\newtheorem{example}[theorem]{Example}
\newtheorem{definition}[theorem]{Definition}

\def\Nesetril{Ne\v{s}et\v{r}il\xspace}
\def\Dvorak{Dvo\v{r}\'{a}k\xspace}
\def\Kral{Kr\'{a}l\xspace}

\newcommand{\str}[1]{{{#1}}}
\newcommand{\from}{\colon}

\newcommand{\set}[1]{\{#1\}}
\newcommand{\setof}[2]{\set{#1\mid#2}}
\def\phi{\varphi}
\def\cal{\mathcal}
\def\N{\mathbb N}
\def\R{\mathbb R}
\def\epsilon{\varepsilon}
\def\eps{\varepsilon}
\renewcommand{\subset}{\subseteq}
\renewcommand{\setminus}{-}
\renewcommand{\le}{\leqslant}
\renewcommand{\ge}{\geqslant}

\newcommand{\dist}{\mathrm{dist}}

\newcommand{\CC}{\cal C}
\newcommand{\DD}{\cal D}

\DeclareMathOperator{\mw}{mw}
\newcommand\reach{{\rm reach}}

\newcommand{\tup}{\bar}
\newcommand{\tp}{\textnormal{tp}}
\newcommand{\ltp}{\textnormal{ltp}}

\newcommand{\stp}{\textnormal{stp}}

\newcommand{\anonym}[2][]{#2}

\newcommand{\ERCagreement}{\xspace ST received funding from the European Research Council (ERC) (grant agreement №948057 -- {\sc bobr} -- and №101126229 -- {\sc buka}).
\begin{tikzpicture}[remember picture, overlay]
  \coordinate (refpoint) at (12,0); 

  \node[anchor=south, yshift=0cm] at (refpoint)
  {\includegraphics[width=40px]{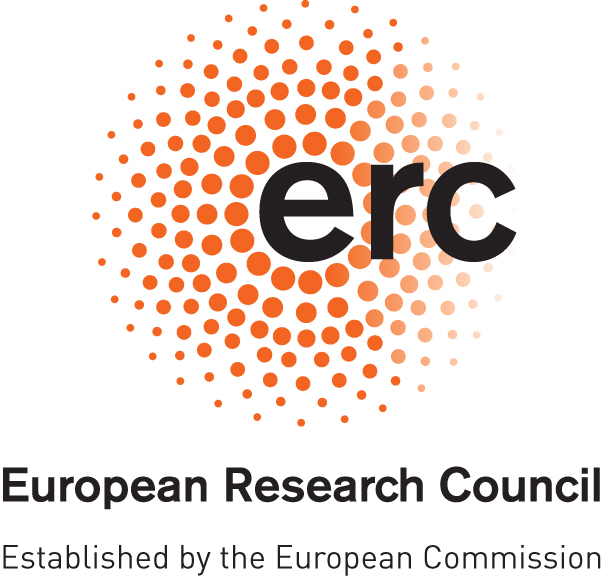}};

  \node[anchor=south, yshift=-2cm] at (refpoint)
  {\includegraphics[width=60px]{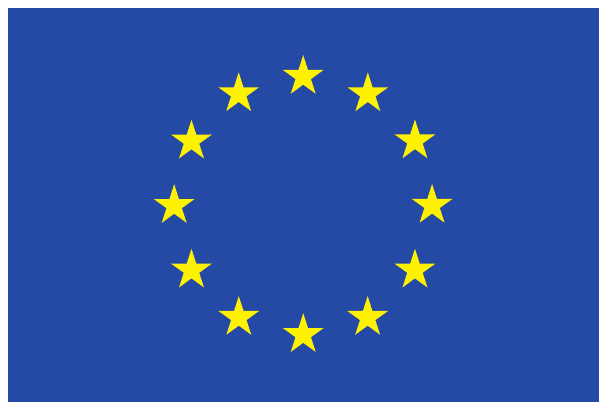}};
\end{tikzpicture}
}

\begin{document}

\title{\vspace{-1.5em}Merge-width and First-Order Model Checking}

\author{{Jan Dreier and Szymon Toru\'nczyk}}

  \date{~}
\maketitle

\vspace{-4em}
\begin{abstract}
We introduce \emph{merge-width}, a family of graph parameters that unifies several structural graph measures, including treewidth, degeneracy, twin-width, clique-width, and generalized coloring numbers.
Our parameters are based on new decompositions called \emph{construction sequences}.
These are sequences of ever coarser partitions of the vertex set, where each pair of parts has a specified default connection, and all  vertex pairs of the graph that differ from the default are marked as \emph{resolved}.
 The \emph{radius-$r$ merge-width} is the maximum number of parts reached from a vertex by following a path of at most $r$ resolved pairs.
 Graph classes of \emph{bounded merge-width}
 -- for which the radius-\(r\) merge-width parameter can be bounded by a constant, for each fixed \(r=1,2,3,\ldots\) --  include all classes of bounded expansion or of bounded twin-width, thus unifying two central notions from the Sparsity and Twin-width frameworks. Furthermore, they are preserved under first-order transductions, which attests to their robustness.
We conjecture that classes of bounded merge-width are equivalent to the previously introduced classes of bounded flip-width.

As our main result, we show that the model checking problem for first-order logic is fixed-parameter tractable on graph classes of bounded merge-width, assuming the input includes a witnessing construction sequence. This unites and extends two previous model checking results: the result of \Dvorak, \Kral, and Thomas for classes of bounded expansion, and the result of Bonnet, Kim, Thomass\'e, and Watrigant for classes of bounded twin-width.
 Finally, we suggest future research directions that could impact the study of structural and algorithmic graph theory, in particular of monadically dependent graph classes, which we conjecture to coincide with classes of \emph{almost bounded merge-width}.
\end{abstract}

\paragraph{Acknowledgements.}
We are grateful to Jakub Gajarsk{\'{y}}, Nikolas M\"ahlmann, Rose McCarty, Jakub No\-wa\-ko\-wski, Pierre Ohlmann,  Michał Pilipczuk, and Wojciech Przybyszewski
for many inspiring discussions. We also thank the anonymous reviewers for numerous useful comments.\ERCagreement

\vspace{-1em}

\section{Introduction}\label{sec:intro}
Several graph parameters in structural graph theory -- in particular, treewidth, degeneracy, and twin-width -- have proven exceptionally effective in both algorithmic and combinatorial contexts.
The study of these parameters has led to a wealth of related measures,
such as generalized coloring numbers (extensions of degeneracy that detect distances up to a fixed radius),
as well as clique-width and rank-width (two closely related extensions of treewidth beyond the realm of sparse graphs).

\medskip
What do all those parameters have in common?
\medskip

They all describe the existence
of graph decompositions of some sort, which can be then
utilized for algorithms and inductive proofs.
However, the decompositions underlying
degeneracy and twin-width, for example,
differ significantly in their structure, utility, and scope.
For instance, there are classes of graphs whose degeneracy is bounded by a constant,
but which have unbounded twin-width, and vice versa.

\smallskip
Ever since the introduction of twin-width, it has been anticipated that twin-width and degeneracy may have some common explanation and a common generalization.
Already with the development
of Sparsity by \Nesetril and Ossona de Mendez,
it has been expected
that its central notions -- of which degeneracy is a prime example -- might have extensions which are suitable beyond the sparse realm, similar to how treewidth is extended by clique-width and rank-width.
Such notions could enable the
analysis of a broader variety of graphs,
and lead to the
unification of Sparsity with Twin-width.
Specifically, two central notions studied in those frameworks are
graph classes of \emph{bounded expansion} -- for which degeneracy, and also each of the generalized coloring numbers is bounded by a constant  --
 and, respectively, graph classes of \emph{bounded twin-width}.
Graph classes of bounded expansion include e.g.\
every  class of bounded maximum degree,
the class of planar graphs,
and every class which excludes some graph as a minor, or even as a topological minor. These graph classes are all sparse: the number of edges is linear in the number of vertices in all graphs from the class.
Graph classes of bounded twin-width
include e.g. every class which excludes some graph as a minor, and also non-sparse graph classes, such as classes of bounded clique-width, proper hereditary classes of permutation graphs, or the class of unit interval graphs. However, some of the simplest classes of bounded expansion -- classes
of bounded maximum degree -- have unbounded twin-width \cite{tww2}. Hence, these two notions are incomparable.
Another central notion studied in Sparsity are \emph{nowhere dense} classes, with coloring numbers bounded by $n^{o(1)}$, for $n$-vertex graphs in the class, rather than $O(1)$. This notion is conceptually similar to, and more general than bounded expansion, and is
discussed at the end of this section.
\Cref{fig:diagram} reviews the different properties of graph classes discussed in this paper.

\paragraph{First-order model checking.}
Classes of bounded twin-width and of bounded expansion share many similarities.
Most notably, similar algorithmic
problems can be efficiently solved on those graph classes, including all problems expressible
in first-order logic.
More precisely, we say that the model checking problem for first-order logic \emph{is fixed-parameter tractable} on a graph class $\CC$ if for every first-order sentence~$\phi$
there is an algorithm which determines whether a given $n$-vertex graph $G\in\CC$ satisfies $\phi$
in time
$c\cdot n^d$, for some constant $c$ which may depend both on $\phi$ and on $\CC$,
and for some exponent $d$ depending on $\CC$ only. Algorithms of this form are known as \emph{algorithmic meta-theorems}, as they establish the fixed-parameter tractability of an entire family of graph problems,  namely those expressible in first-order logic.
This includes the independent set problem, the clique problem, the dominating set problem, and many others.

The model checking problem is fixed-parameter tractable for all graph classes of bounded expansion, by the result of \Dvorak, \Kral, and Thomas \cite{DvorakKT13-journal}.
This also holds for classes of bounded twin-width, by the result of Bonnet, Kim, Thomass\'e, and Watrigant \cite{tww1} -- with the additional requirement that the input graph $G$ is provided together with
a \emph{contraction sequence} witnessing that $G$ has twin-width bounded by a constant.
It remains unknown if there is a polynomial-time algorithm which
computes such a suitable contraction sequence.

\medskip
\paragraph{Flip-width.}
Evidence that the aforementioned graph parameters are
indeed all related
was given by \anonym[Toruńczyk]{the second author} in \cite{flip-width}. It was observed that all parameters discussed so far
can be explained using variants of a pursuit-evasion game, similar to the classic Cops and Robber game that characterizes treewidth \cite{seymour-thomas-cops}.
This led to a
family of graph parameters, called \emph{flip-width}.
Variants of the flip-width parameters recover
 all the graph parameters listed above -- treewidth, twin-width, degeneracy, clique-width, generalized coloring numbers -- up to functional equivalence.
In turn, graph classes of \emph{bounded flip-width} -- in which all the flip-width parameters are bounded --
generalize classes of bounded expansion and classes of bounded twin-width.

Unlike parameters such as treewidth or twin-width,
the flip-width parameters are defined in terms of pursuit-evasion games
and thus do not describe
the existence of a decomposition of the considered graph.
This limits the usefulness of flip-width for algorithmic and combinatorial applications.
In particular, it is not known whether the model checking results
for classes of bounded expansion and for classes of bounded twin-width (with a given contraction sequence) extend
to classes of bounded flip-width.
It is even unknown if in classes of bounded flip-width, the pursuer has a winning strategy
-- the closest equivalent of a decomposition provided by flip-width --
which terminates after a number of rounds which is bounded by a polynomial in the number of vertices, with a fixed degree of the polynomial.

\nocite{robertson-seymour-tw}
\nocite{clique-width}
\nocite{tww2}
\nocite{tww1}
\nocite{grad-and-bounded-expansion-Nesetril}
\nocite{NesetrilM11a}
\nocite{flip-width}
\nocite{Shelah1986}
\begin{figure}
    \centering
    \includegraphics{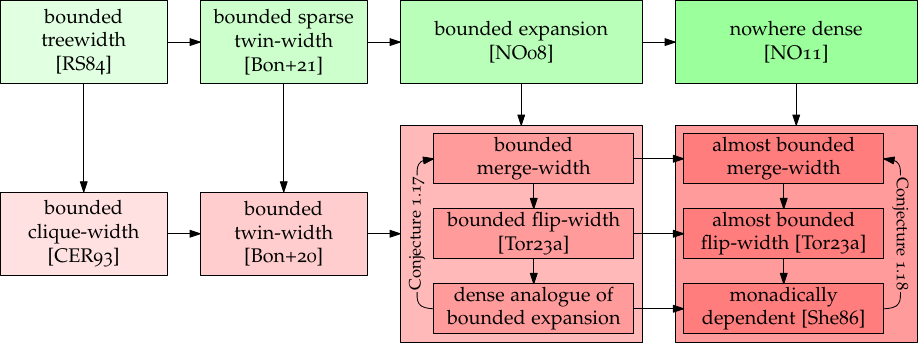}
    \newcommand{\myarrow}[1][]{\mathrel{\tikz{
    \node[fill, minimum width=.5ex, minimum height=.9em, inner sep=0pt, single arrow, single arrow head extend=1.7pt, single arrow tip angle=45]
        (A){\raisebox{3.5pt}[0pt][0pt]{$\,\scriptstyle #1\ $}}; \path([xshift=-.4pt]A.west)--(A.east);}}}
  \caption{Properties of graph classes, and implications ($\myarrow$) among them.
  Each property in the upper row implies the graph class is \emph{weakly-sparse}
  -- excludes some biclique $K_{t,t}$ as a subgraph.
  Each property in the lower row, restricted to weakly sparse graph classes, yields the property  in the upper row directly above it. Also, each property in the lower row, apart
  from almost bounded merge-width/flip-width, is known to be preserved under first-order transductions.
  Within each of the two larger red boxes, we conjecture that also the converse implications hold, and thus that the three notions are equivalent.
  }\label{fig:diagram}
\end{figure}
\subsection*{Contribution}
In this paper, we propose an answer to the question posed in the beginning of this introduction, by directly relating the
\emph{decompositions} underlying the aforementioned parameters. We introduce
\emph{construction sequences}, which are similar to the {contraction sequences} underlying twin-width.
Imposing suitable restrictions on construction sequences allows us to recover contraction sequences underlying twin-width,  tree decompositions underlying treewidth, as well as the decompositions underlying clique-width, degeneracy, and generalized coloring numbers.

As our core conceptual contribution,
we introduce
a new family of graph parameters,
dubbed \emph{\mbox{radius-$r$} merge-width}, for each radius $r\in\N$,
which measure the complexity of graphs in terms of construction sequences.

Our second contribution, and  main result, establishes
that the model checking problem for first-order logic
is fixed-parameter tractable on graph classes of bounded merge-width (for which each merge-width parameter is bounded by a constant),
assuming that the input graph is provided together with a witnessing construction sequence.
These contributions  are detailed further below. We start by defining the main concept of this paper.
\medskip
\begin{tcolorbox}[
  boxsep=2pt,
  left=2pt,
  right=2pt,
  top=2pt,
  ]
\paragraph{Merge-width.} Fix a vertex set $V$.
A \emph{construction sequence} is a sequence of steps, maintaining a partition $\cal P$~of~$V$ and a partition of  ${V\choose 2}$ into three sets:
\emph{edges} $E$, \mbox{\emph{non-edges}~$N$},
and \emph{unresolved} pairs~$U$.
 Initially, $\cal P$ partitions $V$ into singletons, and every pair in $V\choose 2$ is unresolved.
In each step, one of three operations is performed:
\begin{itemize}
  \item  \emph{merge} two parts $A,B\in\cal P$, replacing the two parts by their union $A\cup B$,
  \item \emph{resolve positively} a pair of parts $A,B\in\cal P$
  (possibly $A=B$), declaring
 all the unresolved pairs $\set{a,b}\in U$ with $a\in A,b\in B$ as \emph{edges}, that is, moving them from \(U\) to \(E\), or
 \item \emph{resolve negatively} a pair of parts --
     by declaring the corresponding unresolved vertex pairs as \emph{non-edges}, that is, moving them from \(U\) to \(N\).
\end{itemize}
In the end, we require that $\cal P$ has one part, and that every pair from ${V\choose 2}$ is resolved as either an edge or a non-edge.
We thus say this is a construction sequence of the graph $G=(V,E)$.

\quad The \emph{radius-$r$ width}
of a construction sequence is the least number $k$ such that at every step in the sequence, the following holds:
\begin{quote}
For every vertex $v\in V$, at most $k$ parts
of the current partition $\cal P$ can be reached from $v$ by a path of length  ${\le}r$ in the graph $(V,E\cup N)$ formed by the current edges and non-edges.
\end{quote}
The \emph{radius-$r$ merge-width} of a graph $G$, denoted $\mw_r(G)$, is the least radius-$r$ width of a construction sequence of $G$.
Finally, a graph class $\CC$ has  \emph{bounded merge-width} if $\mw_r(\CC)<\infty$ for all $r\in\N$, where $\mw_r(\CC)\coloneqq\sup_{G\in\CC}\mw_r(G)$.
\end{tcolorbox}
\smallskip

The definition of merge-width in terms of
\emph{construction sequences} is complemented by an equivalent but often more flexible definition via \emph{merge sequences}
 introduced later in \Cref{sec:merge-sequence}.

\begin{remark}\label{remark:maindef}
We start with three observations, regarding each step \((\cal P,E,N,U)\) of a construction sequence of a graph \(G\).

\begin{enumerate}[wide, labelwidth=!, labelindent=0pt]
  \item
        As the edge set of the last step of the construction sequence needs to be \(E(G)\), and
throughout the construction sequence, the sets $E$ and $N$ can never decrease,
        we observe that at each step, \(E \subseteq E(G)\) and \(N \subseteq {V\choose 2} \setminus E(G)\).
        Thus, when the graph \(G\) is clear from the context, at each step, we only need to specify \(\cal P\) and the set \(R\coloneqq E\cup N \subset {V \choose 2}\) of \emph{resolved pairs}.
        The partition of \(V\choose 2\) into \(E, N\), and \(U\) then follows from \(G\), as \(E=R\cap E(G), N = R \setminus E(G)\), and \(U={V\choose 2}-R\).
    \item
        For every pair of parts $A,B\in\cal P$ (including $A=B$), the unresolved pairs $\set{a,b}\in U={V\choose 2}-R$ with $a\in A$ and $b\in B$ are either all adjacent, or are all non-adjacent in~$G$.
        (In other words, the pair $A,B$ is  \emph{homogeneous}, once the vertex pairs in $R$ are ignored.)
        This holds as otherwise, reconstructing the edge set \(E(G)\) in the remainder of the construction sequence would be impossible.
        Thus, for every pair of parts in \(\cal P\), we can declare a \emph{default connection} (either ``adjacent'' or ``non-adjacent'') that
        describes all unresolved vertex pairs between these two parts.
        (If there are no unresolved vertex pairs between them, we may choose an arbitrary default connection.)
    \item
        Assume we want to perform a merge operation between two parts, say \(A,B \in \cal P\).
        Before doing so, for each part \(C\in\cal P\) (including $A$ and  $B$) whose default connection to \(A\) is different than to \(B\), we need to resolve either the pair \(C,A\) or \(C,B\),
        as otherwise the previous item is violated
        (unless one of those pairs is fully resolved).
        All other resolutions can be postponed to after the merge of $A$ and $B$, without increasing the radius-$r$ width of the construction sequence, for each $r\in\N$. 

  \medskip
   See Figures \ref{fig:a}, \ref{fig:b}, and \ref{fig:c} for examples of construction sequences.
\end{enumerate}
\end{remark}
\FloatBarrier
{\definecolor{lightred}{RGB}{255, 158, 159}
\definecolor{darkred}{RGB}{176, 0, 2}
\definecolor{lightergray}{RGB}{234,234,234}
\definecolor{darkergray}{RGB}{126,126,126}
\begin{figure}[p]
\vspace{-0.5cm}
    \centering
    \includegraphics[scale=0.95]{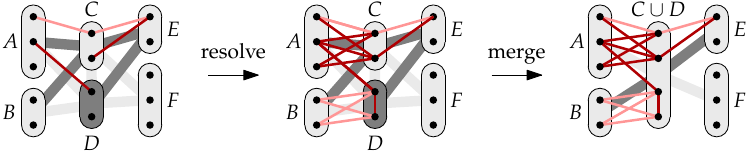}
    \caption{Excerpt of a construction sequence of a graph, with a partition \(\cal P=\{A,B,C,D,E,F\}\).
    The \emph{resolved pairs} are indicated as  \textcolor{darkred}{\rule[0.5ex]{0.3cm}{1.3pt}} (edges) and \textcolor{lightred}{\rule[0.5ex]{0.3cm}{1.3pt}} (non-edges).
    Recall that \emph{unresolved pairs} between any two parts are either all adjacent or all non-adjacent (\Cref{remark:maindef}, item 2).
    This is indicated in an aggregated way as \textcolor{darkergray}{\rule[-0.ex]{0.3cm}{5.3pt}} (edges)
    and \textcolor{lightergray}{\rule[-0.ex]{0.3cm}{5.3pt}} (non-edges) connecting parts,
    or as fill colors within each part.
    Here and below, each arrow labeled ``resolve'' represents a sequence of either positive or negative resolve operations,
    while each arrow labeled ``merge'' usually represents a single merge operation.
    For visual clarity, this figure only shows default connections within or involving parts \(C\)~and~\(D\).
    \smallskip \\
    As discussed in item 3 of \Cref{remark:maindef},
    before we can merge the two parts $C$ and $D$, every part that has a different default connection to $C$ than to $D$ must be fully resolved with either $C$ or~$D$.
    As \(E\) has the same default connection (\textcolor{darkergray}{\rule[-0.ex]{0.3cm}{5.3pt}}) to both \(C\) and \(D\),
    no resolutions involving \(E\) need to be performed before the merge.
    The same holds for~\(F\).
    However, both \(A\) and \(B\) have different default connections to \(C\) and to \(D\).
    Thus, to enable the merge, each of the parts \(A\) and \(B\) need to commit to a default connection to the new part \(C \cup D\).
    The part \(A\) commits to \textcolor{lightergray}{\rule[-0.ex]{0.3cm}{5.3pt}} and thus has to resolve all missing edges to \(C\).
    Similarly, \(B\) commits to \textcolor{darkergray}{\rule[-0.ex]{0.3cm}{5.3pt}} and resolves all missing non-edges to \(D\).
    As \(D\) has a default self-connection
    \textcolor{darkergray}{\rule[-0.ex]{0.3cm}{5.3pt}} before the merge, but \(C \cup D\) commits to  default self-connection \textcolor{lightergray}{\rule[-0.ex]{0.3cm}{5.3pt}},
    the part \(D\) also resolves all missing edges within itself.
    Afterwards, \(C\) and \(D\) can be merged.
}\label{fig:a}
\end{figure}
%
%
\begin{figure}[p]
    \centering
    \includegraphics[scale=0.95]{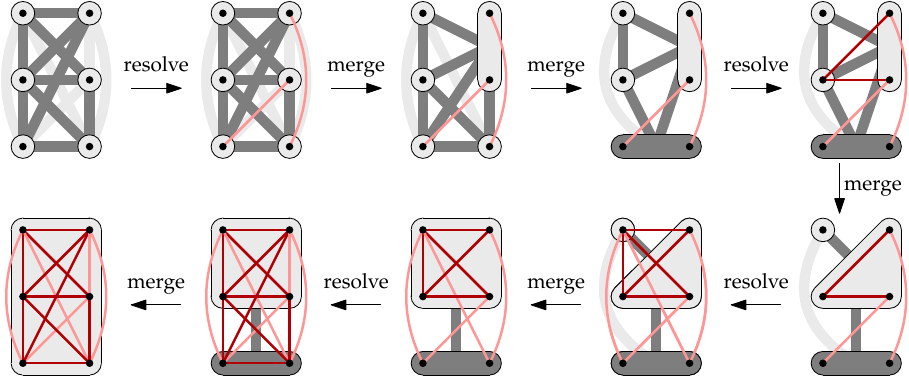}
    \captionsetup{justification=raggedright, singlelinecheck=false}
    \caption{
    A full construction sequence of a graph \(G\) witnessing \(\mw_1(G)\le 3\), using notation from Fig.~\ref{fig:a}.
}\label{fig:b}
\end{figure}
%
\begin{figure}[p]
    \centering
    \includegraphics[scale=0.95]{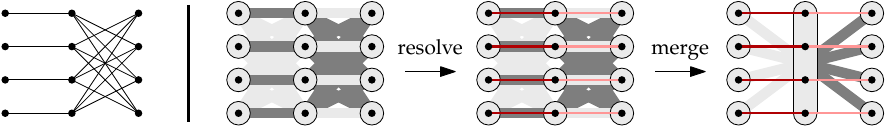}
    \caption{
        \emph{Left:} A graph $G$ consisting of a matching and a complement of a matching.
        \emph{Right:} An excerpt from a construction sequence of $G$.
        First the symmetric difference to the empty default connection (for the matching) and the full default connection (for the complement)
        is resolved, in a sequence of steps. Each vertex reaches three parts by a path of arbitrary length in the resolved graph.
        Then the middle column is merged, in a sequence of steps, into a single part.
        For clarity, default connections are only drawn between neighboring columns.
        Note that in the context of twin-width (see \Cref{ex:tww}), such a contraction sequence would
        result in many parts that are inhomogeneous towards the middle part.
}\label{fig:c}
\end{figure}}
\FloatBarrier

\begin{example}\label{ex:degree}
  As a simple example, fix $r,d\in\N$ and consider a graph $G$ with maximum degree at most $d$. We show that $G$ has radius-$r$ merge-width $O(d^r)$.
  A construction sequence of $G$ proceeds as follows:
  first resolve positively every pair $\set{u},\set{v}$ of singletons such that $u$ and~$v$ are adjacent in $G$. Next, perform a sequence of $|V(G)|-1$ merges, in any order, arriving at a partition with a single part.
  Finally, resolve negatively that part with itself.
  At any moment of the construction sequence (apart from the last step), the set $R$ of resolved vertex pairs is contained in the edge set of $G$; hence, the graph with edge set $R$ has maximum degree $d$.
  It follows that the radius-$r$ width of the construction sequence is at most $1+d+\cdots+d^r\le O(d^r)$ (see \Cref{sec:sparsity} for more details).\qed
\end{example}

\begin{example}\label{ex:tww}
  Again fix $r,d\in\N$ and consider a graph $G$ of \emph{twin-width} at most $d$. We show that $G$ has radius-$r$ merge-width $O(d^r)$.
  Recall that a pair of vertex sets of a graph $G$ is  \emph{homogeneous} if  either all or no edges between these parts are present in $G$.
Further, by definition, $G$ has twin-width at most $d$ if there
is a sequence of steps -- a \emph{contraction sequence} -- maintaining a partition $\cal P$ of $V(G)$, starting with the partition into singletons and ending in the partition with one part, such that in each step some two parts are merged into one, and moreover, each part in $\cal P$ is inhomogeneous towards at most $d$ other parts in $\cal P$. This can be readily converted into a construction sequence, by following the merges of the contraction sequence, and prior to merging two parts $A,B\in \cal P$, resolving the pairs $C,A$, and $C,B$, for each part $C\in\cal P$
such that the pair $C, A\cup B$ is not homogeneous in $G$.
Once $A$ and $B$ are merged, the set $R$ of resolved vertex pairs
consists of those pairs $\set{u,v}$ such that the pair of parts containing $u$ and $v$ respectively is inhomogeneous. It is easy to see
that the radius-$r$ width of the resulting construction sequence is at most $2+d+\cdots+d^r\le O(d^r)$ (see \Cref{sec:tww} for details).\qed
\end{example}

\Cref{ex:tww} proves the following.
\begin{restatable}{theorem}{twwintro}\label{thm:tww}
  Graph classes of bounded twin-width have bounded merge-width.
\end{restatable}

While twin-width measures inhomogeneity on a \emph{part-to-part} level,
the central conceptual novelty of merge-width is
to measure inhomogeneity on a \emph{vertex-to-vertex} level (via the set $R$ of resolved vertex pairs).
This more fine-grained perspective is the key behind
merge-width's additional expressive power (see also Fig.~\ref{fig:c}). For instance,
 graphs of maximum degree at most three have unbounded twin-width \cite{tww2}
and have bounded merge-width, by \Cref{ex:degree}.
As it is expected of any notion that extends concepts from Sparsity theory, the number of reachable parts
needs to be bounded separately for each number of steps \(r \in \N\).
Note that in the definition of a merge sequence, it is crucial that parts are not required to be fully resolved:
if we required that at any time, all vertex pairs contained in a single part are resolved, then 
the resulting parameter for radius $2$ would be functionally equivalent with twin-width.

\begin{example}\label{ex:degeneracy}
  We show that $\mw_1(G)\le d+2$ for all $d$-degenerate graphs.
  A $d$-degenerate graph admits a \emph{degeneracy} total ordering of its vertices, such that every vertex has at most $d$ neighbors before it. Suppose $G$ has $n$ vertices, with degeneracy ordering $v_n<\ldots<v_1$.

  We define a construction sequence consisting of $n$ stages.
  At the beginning of stage $t$, for $t=1,\ldots,n$,
  the current partition  is
$$\cal P_t\coloneqq\Big\{\set{v_n},\ldots,\set{v_{t+1}},\set{v_t,\ldots,v_1}\Big\},$$
and the set of resolved edges \(E_t\) consists of all edges of $G$ with at least one endpoint in $\set{v_t,\ldots,v_1}$,
while the set of resolved non-edges \(N_t\) is empty.
See \Cref{fig:degeneracy} for an illustration.
\begin{figure}
    \centering
    \includegraphics{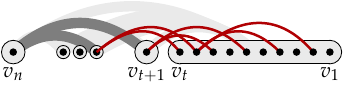}
    \caption{
        A step in a construction sequence of a graph with degeneracy one.
}\label{fig:degeneracy}
\end{figure}
  In stage $t$ with $t<n$, first positively resolve the singleton part $\{v_{t+1}\}$ with all other singleton parts that are adjacent to it in \(G\),
  obtaining the set $E_{t+1}$.
  Afterwards, merge $\set{v_{t+1}}$ with the part $\set{v_t,\ldots,v_1}$,
  obtaining the partition $\cal P_{t+1}$.
  As a last step, in stage $n$, we resolve \(\{v_n,\dots,v_1\}\) negatively with itself,
  guaranteeing that all pairs are resolved.

This is a construction sequence.
  Moreover, at every moment of the construction sequence,
  every vertex $v$ is equal or adjacent in the current resolved graph $(V,E \cup N)$ to vertices in at most $d+2$ parts:
  at most $d+1$ singleton parts (including $\set{v}$),
  and the  part $\set{v_t,\ldots,v_1}$.\qed
\end{example}

Thus, graph classes of bounded degeneracy have bounded radius-1 merge-width
(they may have unbounded radius-2 merge-width, however).
Conversely, it follows from our results (see \Cref{cor:deg}) that
a graph class $\CC$
has bounded degeneracy if and only if it has bounded \mbox{radius-1} merge-width, and
excludes some biclique $K_{t,t}$ as a subgraph.
Thus, radius-1 merge-width may be seen as an extension of degeneracy to graphs which are possibly dense.

On the other extreme, the definition of merge-width also makes sense for the limit radius $r=\infty$.
It is easy to see that $\mw_\infty(G)$ is functionally related to the clique-width of $G$ (see \Cref{thm:cw}).
Thus, in a sense, the merge-width parameters, for finite $r$, are local variants of clique-width.

In this paper, our main focus is the study of graph classes of bounded merge-width.
We already saw in \Cref{thm:tww} that those include all classes of bounded twin-width.
They also include all classes of bounded expansion, a key notion in Sparsity theory.



\begin{restatable}{theorem}{beintro}
  \label{thm:be}
  Graph classes of bounded expansion have bounded merge-width.
\end{restatable}
\noindent The proof of \Cref{thm:be} converts total orders with small \emph{weak coloring numbers} -- a fundamental parameter in Sparsity -- into construction sequences. A graph class $\CC$ has bounded expansion if and only if
for each $r\in\N$ there is some constant $k_r$ such that each graph in $\CC$ has \emph{$r$-weak coloring} number at most $k_r$.
By definition, a graph $G$ with $r$-weak coloring number~$k$ admits a total order $\le$ on its vertices, so that
for every vertex $v\in V(G)$, there are at most $k$ vertices $w\le v$ that can be reached from $v$ by a path $\pi$ of length ${\le}r$, with $w=\min(V(\pi))$.
We use such an order to prove $\mw_{r-1}(G)\le O(2^k)$, by defining a construction sequence of $G$ of radius-$(r-1)$ width at most $O(2^k)$.
Similarly as in \Cref{ex:degeneracy}, the sequence proceeds in $n$ stages. At the beginning of stage~$t$, the partition $\cal P_t$ partitions the set $S_t$
of $n-t$ smallest elements of $V$ into singletons, and the remaining vertices are partitioned according to their neighborhoods in $S_t$, whereas the set $E_t$ of resolved edges consists of all edges of $G$ whose \emph{both endpoints} (unlike in \Cref{ex:degeneracy}) lie outside of $S_t$. A simple analysis shows that the ball of radius $(r-1)$ around any given vertex in the graph $(V,E_t)$ intersects at most $2^k$ parts of $\cal P_t$.
See \Cref{lem:wcol} for details.


\medskip
On the other hand, classes of bounded merge-width have bounded flip-width.

\begin{theorem}\label{thm:mw-fw}
Every class of bounded merge-width has bounded flip-width.
\end{theorem}
\noindent\Cref{thm:mw-fw}, proved in \Cref{sec:fw}, allows us to derive properties of classes of bounded merge-width from known properties of classes of bounded flip-width. In particular, we get the following corollaries,
characterizing classes of bounded expansion and of bounded twin-width in terms of merge-width.

\begin{restatable}{corollary}{corbe}\label{cor:be}
  A graph class has bounded expansion if and only if it has bounded merge-width, and is weakly sparse (excludes some biclique $K_{t,t}$ as a subgraph).
\end{restatable}

In order to characterize twin-width in terms of merge-width, note that  the definition of merge-width extends easily to structures equipped with one or multiple binary relations, such as \emph{ordered graphs} --  graphs equipped with a total order on the vertex set. (See also \Cref{sec:computing} for merge-width of binary structures.) Similarly, the notion of twin-width also applies to binary structures \cite{tww4}, and \Cref{thm:tww} and \Cref{thm:mw-fw} also hold for classes of binary structures. We get the following corollary, from an analogous result concerning flip-width \cite[Thm. II.8]{flip-width}.
\begin{corollary}\label{cor:tww-mw-ordered}
  A class of ordered graphs has bounded twin-width if and only if it has bounded merge-width.
\end{corollary}
The following is a consequence for usual (unordered) graphs, as every graph can be equipped with a total order, yielding an ordered graph of the same twin-width \cite{tww4}.

\begin{restatable}{corollary}{cortww}\label{cor:tww-mw}
  A graph class has bounded twin-width if and only if it is obtained
  from some class of {ordered graphs} of bounded merge-width
  by forgetting the order.
\end{restatable}

We conjecture that
the converse to \Cref{thm:mw-fw} also holds: that every class of bounded flip-width has bounded merge-width (see Discussion below).
As we demonstrate in \Cref{sec:fw}, construction sequences can be
viewed as representations of specific winning strategies for the pursuer in the flip-width game. In particular,
in those strategies, the pursuer wins in $n$ rounds on an $n$-vertex graph (see \Cref{rem:duration}).
This is reminiscent of the situation with treewidth, where tree decompositions correspond to \emph{monotone} winning strategies for the pursuer in the Cops and Robber game.
If indeed every class of bounded flip-width
has bounded merge-width, this would imply that whenever
the pursuer has a winning strategy, then
they also have one of a special form, analogously to the treewidth case.

\paragraph{Main result.}
As our second contribution and  main result,
we prove that the model checking problem for first-order logic
is fixed-parameter tractable on graph classes of bounded merge-width,
provided the input graph comes with a construction sequence demonstrating that its merge-width is bounded.
More precisely, we prove:

\begin{restatable}{theorem}{intromain}\label{thm:main}
There is an algorithm which,
given a graph $G$ together with a construction sequence
and a first-order graph sentence $\phi$,
determines whether $G$ satisfies $\phi$
in time
\(O_{q,w}(|V(G)|^3)\),
where:
\begin{itemize}
    \item $q$ is the quantifier rank of $\phi$,
    \item $w$ is the radius-\(r\) width of the construction sequence, where  $r=2^{O(q^2)}$.
\end{itemize}
\end{restatable}
\noindent 
See \Cref{sec:representing} for how construction sequences are represented.

\medskip
\Cref{thm:main} conceptually generalizes two previous results concerning classes of bounded expansion and classes of bounded twin-width:
\begin{enumerate}[wide, labelwidth=!, labelindent=0pt]
  \item We extend the result of \Dvorak, \Kral, and Thomas~\cite{DvorakKT13-journal}, that model checking is fixed-parameter tractable on all graph classes of bounded expansion. This is because for every graph $G$ from such a class~$\CC$, a construction sequence of bounded radius-$r$ width can be computed in  time\footnote{Throughout this paper,
  by $O_{p}(n)$, where $p$ is a list of parameters, we denote a value which is upper-bounded by $c\cdot n+d$, for some constants $c,d$ depending on $p$.} $O_{r,\CC}(|V(G)|)$, for every fixed $r\in\N$
  (see \Cref{thm:be}).

  \item We also extend the result of Bonnet, Kim, Thomass\'e, and Watrigant \cite{tww1}, that model checking is fixed-parameter tractable on all graphs of bounded twin-width, assuming that
  the input graph $G$ is given together with a contraction sequence witnessing that the twin-width is bounded by a constant $d$. This is because such a contraction sequence can be converted in time $O_d(|V(G)|)$ into a construction sequence of radius-$r$ width at most $O(d^r)$, for every $r\in\N$ (see \Cref{ex:tww}).
\end{enumerate}
Note that in these two results, the running time is linear in $|V(G)|$, whereas in \Cref{thm:main}
it is cubic, and thus, our result falls short of fully implying them.
We conjecture that \Cref{thm:main} can be strengthened to obtain a linear dependence on \(|V(G)|\).

Our proof combines tools used in the context of model checking for classes of bounded twin-width (\emph{local types} \cite{tww-types-icalp}),
with tools used in the context of model checking for nowhere dense classes (\emph{rank-preserving Gaifman normal form} \cite{gks}).
See \Cref{sec:overview} for an overview of our proof.

\paragraph{Closure under interpretations and transductions.}
In addition,
we exhibit some further properties of classes of bounded merge-width.
Namely, they are preserved under one-dimensional first-order interpretations (more generally, transductions), a notion from model theory.
This roughly means that redefining the edges of each graph in the class using a fixed first-order formula $\phi(x,y)$
preserves classes of bounded merge-width.
More precisely, for a graph  $G$ and first-order formula  $\phi(x,y)$,
we construct
a graph $\phi(G)$
with vertices $V(G)$ and edges $\set{u,v}\in {V(G)\choose 2}$ such that $G\models\phi(u,v)\lor\phi(v,u)$.
In \Cref{sec:closure} we prove the following:

\begin{restatable}{theorem}{introinterp}\label{thm:interp}
  Let $\phi(x,y)$ be a first-order formula and
   $\CC$ be a graph class of bounded merge-width.
  Then the class $\setof{\phi(G)}{G\in\CC}$ has bounded merge-width.
\end{restatable}

For example, with \(\phi(x,y)=\exists z \,E(x,z) \land E(y,z)\),
this shows that the class of squares of graphs from a fixed class of bounded merge-width again forms a class of bounded merge-width, which is already a non-trivial fact.
The theorem moreover implies that classes
of \emph{structurally bounded expansion} \cite{lsd-journal}
-- classes of the form $\phi(\CC)$, for a first-order formula $\phi(x,y)$ and  class $\CC$ of bounded expansion -- also have bounded merge-width.

More generally, a graph class $\CC$ \emph{transduces} a graph class $\DD$,
if there is a formula $\phi(x,y)$ involving the binary edge relation symbol and additional unary relation symbols,
such that every graph $H\in\DD$ can be obtained from some graph $G\in\CC$
in the following steps:
\begin{enumerate}
  \item expand $G$ arbitrarily by unary relations appearing in $\phi$, obtaining a structure~$G'$,
  \item pick an arbitrary subset $W\subset V(G)$,
  \item obtain a graph $H$ with vertices $W$ and edges $\set{u,v}\in {W\choose 2}$ such that $G'\models\phi(u,v)$.
\end{enumerate}
\noindent
The following is essentially a rephrasing of \Cref{thm:interp}, in a slightly more expressive language.
\begin{corollary}\label{cor:transduce}
  Suppose $\CC$ and $\DD$ are graph classes, and that $\CC$ transduces $\DD$.
  If $\CC$ has bounded merge-width, then so does $\DD$.
\end{corollary}

\paragraph{Almost bounded merge-width.}
Apart from classes of bounded expansion, another -- arguably even more important --
concept studied in Sparsity, are nowhere dense classes.
A graph class $\CC$ is \emph{nowhere dense} if for every $r\in\N$ there is some $k\in\N$
such that the $r$-subdivision  of the $k$-clique (obtained by replacing each of its edges with a path with $r$ inner vertices) is not a subgraph of any graph in $\CC$.
Those classes strictly extend classes of bounded expansion, and each such class excludes some biclique $K_{t,t}$ as a subgraph. It follows from \Cref{cor:be}
that nowhere dense classes and classes of bounded merge-width are incomparable
(see \Cref{fig:diagram}).

Nowhere dense graph classes can be characterized in terms of generalized coloring numbers:  A~hereditary\footnote{closed under taking induced subgraphs} graph class is nowhere dense if and only if for each fixed $r\in\N$, the $r$-weak coloring number of every $n$-vertex graph in the class is bounded by $n^{o(1)}$ (relaxing the constant bound in classes of bounded expansion).
Motivated by this characterization, we define the following notion.
A graph class has \emph{almost bounded merge-width} if for each fixed ${r\in\N}$,  every $n$-vertex graph in the class has radius-$r$ merge-width at most $n^{o(1)}$.
Trivially, classes of almost bounded merge-width
include all classes of bounded merge-width. Furthermore, they
 include all nowhere dense graph classes (see \Cref{sec:abmw}):

\begin{restatable}{theorem}{thmnwd}\label{thm:nwd}
  Every nowhere dense graph class has almost bounded merge-width.
\end{restatable}

Classes of \emph{almost bounded flip-width} are defined \cite{flip-width} analogously
to classes of almost bounded merge-width,
with the merge-width parameters replaced with the flip-width parameters.
The following result parallels \Cref{thm:mw-fw}:
\begin{restatable}{theorem}{thmabmw}\label{thm:abmw}
  Every hereditary class of almost bounded merge-width has almost bounded flip-width.
\end{restatable}

For a hereditary, weakly sparse graph class (excluding some biclique as a subgraph), nowhere denseness is equivalent to almost bounded flip-width \cite[Thm. X.8]{flip-width}. With \Cref{thm:nwd} and \Cref{thm:abmw}, this implies the following.
\begin{corollary}\label{cor:nd}
  The following conditions are equivalent for a hereditary, weakly sparse graph class $\CC$:
  \begin{enumerate}
    \item $\CC$ is nowhere dense,
    \item $\CC$ has almost bounded merge-width,
    \item $\CC$ has almost bounded flip-width.
  \end{enumerate}
\end{corollary}

We conjecture that in general, for hereditary graph classes almost bounded merge-width, almost bounded flip-width, and \emph{monadic dependence}
 coincide
(see \Cref{conj:almostboundedmergewidth} below).

\subsection*{Discussion}
We discuss several possible research directions and open problems.

\paragraph{Combinatorial properties.}
Classes of bounded twin-width and classes of structurally bounded expansion  constitute two prime examples
of classes of bounded merge-width.
Those classes are known to share several combinatorial properties, such as
(polynomial) $\chi$-boundedness, the strong Erd{\H o}s-Hajnal property, or admitting $O(\log(n))$-adjacency labelling sche\-mes \cite{tww2, tww3, lsd-journal}.
We conjecture that classes of  bounded merge-width also enjoy all those properties.

On the other hand, we conjecture that if a hereditary graph class is \emph{small} -- contains at most $n!\cdot 2^{O(n)}$ distinct labeled $n$-vertex graphs -- then it has bounded merge-width. This would extend (by \Cref{cor:be}) a result
\cite{DBLP:conf/innovations/BonnetD0Z25}
 that every hereditary, weakly sparse, small graph class has bounded expansion,
and (by \Cref{cor:tww-mw-ordered}) a result 
\cite{tww4} that every small, hereditary class of ordered graphs has bounded twin-width.
This would also strengthen a result
\cite{braunfeld2023existentialcharacterizationsmonadicnip,flip-breakability} that small, hereditary classes are monadically dependent (see definition below).
If indeed all classes of bounded merge-width admit $O(\log(n))$-adjacency labelling schemes, this would also imply that every hereditary, small graph class admits such a labelling scheme, as conjectured in \cite{bonnet_et_al:LIPIcs.ICALP.2024.31}.

\paragraph{Algorithmic simplicity.}
 Our model-checking algorithm of \Cref{thm:main} relies on technically involved machinery for first-order logic.
 We are curious whether fundamental problems such as \emph{independent set} or \emph{dominating set}
 can be solved on graph classes of bounded merge-width via more direct techniques, and more efficiently --
 in particular, with linear time dependency.

\paragraph{Approximating merge-width.}
A burning issue is the question of obtaining
an fpt approximation algorithm for merge-width.
More precisely, to prove fixed-parameter tractability of the first-order model checking problem over all classes of bounded merge-width,
it would be sufficient to show that for every class $\CC$ of bounded merge-width and $r\in\N$, there is a bound $w$ and an algorithm
which, given a graph $G\in\CC$, outputs
its construction sequence of radius-$r$ width $w$,
 in time $O_{\CC,r}(|V(G)|^c)$, for some universal constant $c$.
 Proving this seems to be challenging.
 Such an approximation algorithm exists
 in two special cases:
 \begin{itemize}
  \item For classes $\CC$ of graphs of bounded merge-width which exclude some biclique $K_{t,t}$ as a subgraph,
  that is, classes of bounded expansion, by  \Cref{cor:be}.
  For those classes, an approximation algorithm for the weak coloring numbers follows from the results of \cite{dvorak-admissibility}. This allows to approximate the merge-width parameters, as follows from the proof of \Cref{thm:be}.
  \item
  For classes $\CC$ of ordered graphs of bounded merge-width
  (more generally, classes with \emph{effectively} bounded twin-width). Those coincide with classes of ordered graphs of bounded twin-width, by \Cref{cor:tww-mw-ordered}.
  For those classes, an approximation algorithm for twin-width exists, by the result of \cite{tww4}. This allows to approximate the merge-width parameters, as follows from the proof of \Cref{thm:tww}.
 \end{itemize}

 In both cases, the algorithms proceed greedily and construct a suitable decomposition, unless  an obstruction to the existence of such a decomposition is encountered. This suggests
 that, in order to approach the problem in general, a suitable understanding of the obstructions to having small merge-width is required. A conjectured, concrete form of such obstructions is discussed below.

 For this reason, an fpt approximation algorithm for merge-width could turn out to be more approachable than an fpt approximation algorithm for twin-width, as even for classes of cubic graphs (in particular, $K_{4,4}$-free graphs),
 a characterization of classes of bounded twin-width in terms of forbidden obstructions seems elusive.

 \paragraph{Merge-width and flip-width.}
As mentioned, we conjecture that classes of bounded merge-width coincide with classes of bounded flip-width. Besides simplifying the picture in \Cref{fig:diagram}, we expect that a proof of this conjecture might lead to an fpt approximation algorithm for merge-width. To approach  both questions, the missing piece might be a duality result, stating the equivalence of the existence of decompositions in the form of construction sequences, and the lack of certain obstructions.
Indeed, both in the sparse setting, and in the ordered setting (see two bullet points in previous paragraph), such a duality result yields both the approximation algorithm as well as the equivalence with flip-width \cite{flip-width}.
Concrete obstructions that might characterize classes of bounded flip-width, and therefore possibly also classes of bounded merge-width, were conjectured in
\cite{flip-width}, in the form of \emph{hideouts},
and -- in a more explicit form -- in \cite[Sec. 4]{flip-breakability}.
Without going into the particulars of those obstructions, we pose the following conjecture.
Following \cite{tww-types-icalp}, we say that a graph class $\CC$ is in the \emph{dense analogue of bounded expansion} if
 for every weakly sparse graph class $\DD$ such that $\CC$ transduces $\DD$,
 the class $\DD$ has bounded expansion (equivalently in this context, by \cite{Dvorak18}, $\DD$ has bounded degeneracy).


\begin{conjecture}\label{conj:bmw}The following conditions are equivalent for a graph class $\CC$:
  \begin{enumerate}
    \item\label{conj1:it1} $\CC$ has bounded merge-width,
    \item\label{conj1:it2} $\CC$ has bounded flip-width,
    \item\label{conj1:it3} $\CC$ is in the dense analogue of bounded expansion.
  \end{enumerate}
\end{conjecture}
The implication \eqref{conj1:it1}$\rightarrow$\eqref{conj1:it2} is given by \Cref{thm:mw-fw} above.
The implication \eqref{conj1:it2}$\rightarrow$\eqref{conj1:it3} is proved in \cite[Thm. VI.3]{flip-width}.
The implication \eqref{conj1:it3}$\rightarrow$\eqref{conj1:it2} is conjectured in \cite[Conj. XI.7.]{flip-width}.
We propose the stronger conjecture, \eqref{conj1:it3}$\rightarrow$\eqref{conj1:it1}.

 \Cref{conj:bmw} holds in two settings:
 \begin{itemize}
  \item for graph classes $\CC$ that exclude some biclique $K_{t,t}$ as a subgraph, due to \Cref{cor:be},
  \item for classes $\CC$ of ordered graphs, due to \Cref{cor:tww-mw-ordered}.
 \end{itemize}

\paragraph{Model checking and monadic dependence.}
The celebrated result of Grohe, Kreutzer, and Siebertz \cite{gks} establishes the fixed-parameter tractability
of the model checking problem
for all nowhere dense classes.
Furthermore, for a weakly sparse graph class $\CC$,
the model checking problem is fixed-parameter tractable on $\CC$ if and only if $\CC$
is nowhere dense (under standard complexity theoretic assumptions).
Similarly, for a hereditary class of \emph{ordered} graphs (see \Cref{cor:tww-mw}),
the model  checking problem is fixed-parameter tractable on $\CC$ if and only if $\CC$ has bounded twin-width \cite{tww4}.
The result of Grohe, Kreutzer, and Siebertz  was subsequently extended to \emph{monadically stable} graph classes \cite{ms-mc1,ms-mc2}.

\emph{Monadic dependence}, which subsumes all those graph classes,
 is a notion which was introduced by Shelah in his momentous classification program in model theory, and is defined as follows:
A graph class $\CC$ is monadically dependent if and only if it does not transduce the class of all graphs.
Monadically dependent classes include all graph classes that
have been discussed so far, including classes of almost bounded merge-width and almost bounded flip-width (see \Cref{fig:diagram}).

\medskip
A central conjecture in algorithmic model theory -- which we call the \emph{model checking conjecture} here -- states that the model checking problem for first-order logic is fixed-parameter tractable for precisely those hereditary graph classes which are
monadically dependent.
Currently,
monadically dependent classes lie beyond the reach of algorithmic methods,
primarily due to the lack of suitable decompositions, similar to the ones
used in the results for nowhere dense, monadically stable, or bounded twin-width graph classes.

We conjecture that monadically dependent classes have almost bounded merge-width.
More precisely, we pose the following.


\begin{conjecture}\label{conj:almostboundedmergewidth}
The following conditions are equivalent for a hereditary graph class $\CC$:
\begin{enumerate}
  \item\label{conj2:it1} $\CC$ has almost bounded merge-width,
  \item\label{conj2:it2} $\CC$ has almost bounded flip-width,
  \item\label{conj2:it3} $\CC$ is monadically dependent.
\end{enumerate}
\end{conjecture}

The implication \eqref{conj2:it1}$\rightarrow$\eqref{conj2:it2} is established in \Cref{thm:abmw}.
The implication \eqref{conj2:it2}$\rightarrow$\eqref{conj2:it3} was proved in \cite{flip-breakability}. The implication \eqref{conj2:it3}$\rightarrow$\eqref{conj2:it2} was conjectured in \cite{flip-width}. We strengthen this conjecture, by conjecturing the implication \eqref{conj2:it3}$\rightarrow$\eqref{conj2:it1}.
This would generalize \Cref{cor:nd}, as for weakly sparse graph classes, monadic dependence coincides with nowhere denseness.

It seems that a proof of this last implication
could provide a key to achieving the missing
decompositions for monadically dependent classes, and provide a crucial step towards the resolution of the model checking conjecture.

\section{Proof overview}\label{sec:overview}
We now discuss the proof of our main result, \Cref{thm:main}.
Our model checking algorithm iteratively computes for each step of a given construction sequence local first-order information.
More precisely,
at any stage of a construction sequence,
we have a structure $A=(V,E,N,\cal P)$,
where $\cal P$ is a partition of $V$ and $E,N\subset {V\choose 2}$ are two disjoint subsets.
We view $A$ as a logical structure by representing each part of the partition $\cal P$ as a separate unary predicate. Thus, the signature of $A$ has two binary predicates and up to $|V|$ unary predicates. However, crucially, every vertex $v$ only reaches a bounded number of distinct unary predicates by paths of length ${\le}r$ in the graph $(V,E\cup N)$, assuming the construction sequence has bounded radius-$r$ width.

 We then employ a key property of first-order logic, called \emph{locality}, which roughly says that, for a fixed structure $A$ and formula $\phi(x)$, whether or not $\phi(x)$ holds at a given vertex $v\in V(A)$ depends only on the neighborhood of $v$ in the Gaifman graph of $A$, of radius depending on the quantifier rank of $\phi$. In our case, the Gaifman graph of the structure $A=(V,E,N,\cal P)$ is just the graph with the resolved pairs, $(V,E\cup N)$.

     Let us give more details.
     Define  the \emph{\(q\)-type} of a vertex \(v\) in a structure \(A'\) as the set of all formulas \(\phi(x)\) of quantifier rank up to \(q\)
     such that \(\phi(v)\) holds in \(A'\).
     Further, define the \emph{local reduct around \(v\)} as the structure \(A'\) obtained from \(A\)
     by keeping only those unary predicates which can be reached from \(v\) in the Gaifman graph via a path of length bounded by a fixed constant depending on $q$.
     See \Cref{fig:reduct} for an illustration.

\begin{figure}
    \centering
    \includegraphics{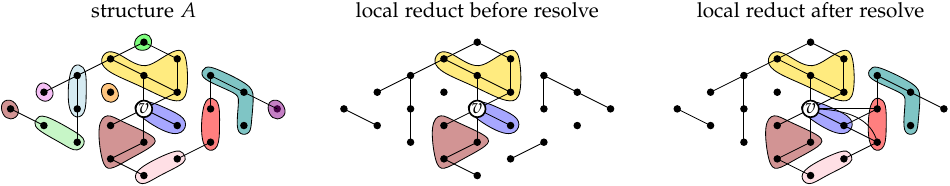}
    \caption{
        \emph{Left:} The structure \(A\). Each color describes a separate unary predicate and black edges describe the Gaifman graph.
        \emph{Center:} The local reduct around \(v\) contains only unary predicates reachable via two steps from \(A\).
        \emph{Right:} The local reduct around \(v\) contains more predicates after a resolve operation (between red and blue).
        Incorporating these additional predicates into the local type is our key challenge.
    }\label{fig:reduct}
\end{figure}

     As an invariant for our model checking algorithm, we compute at each step of a construction sequence, and each vertex \(v\),
     its \(q\)-type in the local reduct around \(v\).
     We call this \emph{the local \(q\)-type} of~\(v\).
     The local types in the final step of the construction sequence then give the answer to the model checking query, as the sequence finishes with all edges of $G$ resolved.
     Crucially, as the merge-width bounds the number of unary predicates in the local reducts,
     this also bounds the number of formulas to consider for the local types, and thus ensures the efficiency of our algorithm.

     The local types are trivial to compute for each vertex in the beginning of the sequence, as there are no edges and non-edges,
     and therefore all vertices carry the same information.
     The challenge is to show that this information can be updated
     with each step of the construction sequence.

     Dealing with merge operations is easy, since those operations only simplify the considered structure,
     by replacing two unary predicates by their union, without affecting the binary predicates.
     The essential difficulty is to deal with resolve operations, as those alter the binary relations,
     and therefore modify the Gaifman graph of the structure.
     Intuitively, suddenly a vertex $v$ may see a lot more structure in its vicinity compared to the previous step.
     Given the local \(q\)-type of a vertex \(v\), computed in a reduct \(A'\) of \(A\) (\Cref{fig:reduct}, center), the difficulty therefore lies
     in computing the \(q\)-type in a richer reduct \(A''\) of \(A\)
     that also contains the unary predicates of \(v\)'s new surroundings (\Cref{fig:reduct}, right).

     While Gaifman's classical theorem implies that the local type of a vertex determines its type in the larger reduct \(A''\), it has one crucial drawback:
     To obtain the \(q\)-type in \(A''\), we need to know beforehand the local \(q^*\)-type for some \(q^* > q\).
     Thus, to compute the local \(q\)-type for the last step, we would have to compute local \(q^*\)-types early on for unbounded values of \(q^*\), which is infeasible.

     In a standalone note~\cite{locality-note} we prove a locality theorem (\Cref{thm:localglobal}), that determines the \emph{\(q\)-type} of a vertex from its \emph{local \(q\)-type} (as well as some global information).
     The groundbreaking work of Grohe, Kreutzer and Siebertz introduced a similar locality theorem
     under the name \emph{rank-preserving normal form}~\cite{gks}.
     Our construction takes great inspiration from this.
     In particular, we also extend first-order logic with additional distance predicates, and rely on similar proof techniques.
     But our result also  differs significantly.
     First, our locality result considers a weaker form of \emph{scatter sentences}, which is easy to evaluate algorithmically (evaluating the scatter sentences of \cite{gks} is as hard as the distance-$r$ independent set problem, which is therefore solved separately in \cite{gks}).
     Second, as Grohe, Kreutzer, and Siebertz applied their locality theorem in the context of a bounded-depth recursion,
     it was permissible to add additional unary predicates to the structure in which the local type is computed.
     As we propagate local information more often, once for each step of the construction sequence,
     we cannot rely on additional unary predicates and therefore cannot use the result of~\cite{gks}.
     For the same reason, the locality theorem of~\cite{ms-mc1} developed in the context of monadically stable classes is also not suitable for our purposes.
     We believe our new locality theorem will serve as a crucial tool in the study of monadically dependent graph classes.

 \paragraph{Outline.}
 In \Cref{sec:prelims}, we start by giving a rigorous definition of merge-width, and prove some basic equivalences.
 Then in \Cref{sec:logic}, we state the locality theorem and the local-subgraph lemma from the standalone locality note~\cite{locality-note}, and in \Cref{sec:mc} we incorporate them in our model checking algorithm.
 Finally, in \Cref{sec:closure} and \Cref{sec:cases} we exhibit closure properties of merge-width and its relationship
 to other structural parameters; we also study classes of almost bounded merge-width.

\input{prelims}

\input{proofs}

\input{mc}

\input{cases}
\printbibliography

\end{document}

%% file: prelims.tex
\section{Preliminaries}\label{sec:prelims}

We now state formally the main definitions of this paper, and prove some basic equivalences. 

       \paragraph{Notation, graphs.}
       For two distinct elements $a,b$ of a set $V$,
by $ab$ we denote the unordered pair $\set{a,b}$, and call $ab$ a \emph{pair} for simplicity
-- the distinction between ordered pairs and unordered pairs will follow from the context.
For $A,B\subset V$, denote $$AB\coloneqq \setof{ab}{a\in A,b\in B,a\neq b}$$ and ${A\choose 2}\coloneqq AA$.
A \emph{graph} is a pair $G=(V,E)$ with $E\subset {V\choose 2}$.
We call the elements of $V$ \emph{vertices}, and elements of $E$ \emph{edges}, and denote $V(G)\coloneqq V$ and $E(G)\coloneqq E$. 

\subsection{Merge-width and merge sequences}\label{sec:merge-sequence}

In the gray box of \Cref{sec:intro} we gave a definition of merge-width in terms of \emph{construction sequences}
(see also \Cref{sec:computing} for a more formal definition, in the more general setting of binary structures).
We now give an alternative definition in terms of so-called \emph{merge sequences}, and prove in \Cref{lem:fmw}
that these definitions are indeed equivalent.
These two definitions complement each other: 
While construction sequences are useful for showing that a graph class with bounded merge-width has a certain property,
merge sequences are useful for the opposing goal of showing that a given graph class has bounded merge-width.

\medskip


    Fix $r\in\N\cup\set{\infty}$.
        For a set $R\subset {V\choose 2}$ and partition $\cal P$ of $V$,
        we say a part \(A \in \cal P\) is \emph{\(r\)-reachable} from a vertex \(v\) if there is a path of length at most $r$ in the graph $G=(V,R)$ from $v$ to some vertex in~$A$.
        The \emph{radius-$r$ width} of $(\cal P,R)$  is 
        \[
           \max_{v \in V}\ \Big|\bigl\{ A \in \cal P \mid  A \text{ is \(r\)-reachable from } v \bigr\}\Big|.
        \]

\begin{definition}\label{def:ms}
A \emph{merge sequence} of a graph $G=(V,E)$ 
is a sequence
$$(\cal P_1,R_1),\ldots,(\cal P_m,R_m),$$
  such that:
  \begin{enumerate}
    \item $\cal P_1,\ldots,\cal P_m$ is a sequence of partitions of $V$, 
    where $\cal P_1$ is the partition into singletons and $\cal P_m$ has one part,
    and $\cal P_{t}$ is coarser than (or equal to) $\cal P_{t-1}$ for $t=2,\ldots,m$,
    \item $R_1\subset\ldots\subset R_m\subset {V\choose 2}$; the pairs in $R_t$ are said to be \emph{resolved} at time $t$, and
    \item for all $t\in [m]$ and $A,B\in \cal P_t$,
     either $AB-R_t\subset E$, or $AB-R_t\subset {V\choose 2} - E$.
  \end{enumerate}
For $r\in\N\cup\set{\infty}$, the \emph{radius-$r$ width} of the merge sequence is the maximum, over $t=2,\ldots,m$,
of the radius-$r$ width of $(\cal P_{t-1},R_t)$.
\end{definition}

Note the offset in the indices of $\cal P_{t-1}$ and $R_t$.
Conceptually, this offset allows us to not only do a single merge, but an unbounded number of merges when going from partition \(\cal P_{t-1}\) to \(\cal P_t\) (see backwards direction of \Cref{lem:fmw}).

\begin{definition}
  The \emph{radius-$r$ merge-width} of a graph $G$, denoted $\mw_r(G)$, is the least radius-$r$ width of a merge sequence of $G$.
Finally, a graph class $\CC$ has \emph{bounded merge-width}
  if ${\mw_r(\CC)<\infty}$ holds\footnote{For a graph parameter $f$ and graph class $\CC$, by $f(\CC)$ we denote $\sup_{G\in\CC}f(G)$.} for every (finite) $r\in\N$.
\end{definition}

\begin{lemma}\label{lem:fmw}
    For every construction sequence of a graph $G$ there is a merge sequence of $G$
  of the same radius-$r$ width, for all $r\in\N\cup\set\infty$,
  and vice versa.
  \end{lemma}
  \begin{proof}
  
  First, fix a construction sequence of $G$, consisting of $m$ steps. 
    Let $\cal P_t$ be the partition of $V$,
    and let $R_t$ be the set of edges and non-edges 
    at the beginning of step $t$ of the sequence.
    It is clear that 
    $(\cal P_1,R_1),\ldots,(\cal P_m,R_m)$ is a merge sequence. 
    In particular, the third item of \Cref{def:ms} follows from the second item of \Cref{remark:maindef}.
    
    Fix $r\in\N\cup\set{\infty}$, and let $k$ 
    be the radius-$r$ merge-width of the considered construction sequence.
    Then, for each $t\in[1,m]$, 
    the radius-$r$ width of $(\cal P_t,R_t)$ is at most $k$.
    Since for $t>1$, we either have $\cal P_{t-1}=\cal P_t$ or $R_{t-1}=R_t$, it follows that 
    the radius-$r$ width of $(\cal P_{t-1},R_t)$ is at most $k$ as well.
    This concludes one direction.

  \medskip
  
  In the other direction,  
  let $(\cal P_1,R_1),\ldots,(\cal P_m,R_m)$ be a merge sequence of $G$.
  We define a construction sequence which consists of $m$ stages, as follows.
  In the beginning of stage $t\in[m-1]$,
  the current partition $\cal P$ is the partition $\cal P_t$, and the set $R$ of resolved pairs is contained in $R_t$.
  In stage $t$, first resolve (in any order) all pairs 
  $A,B\in\cal P_{t}$ such that $AB\subset R_{t+1}$.
  The third item of \Cref{def:ms} guarantees that this is always possible by positive or negative resolves.

  Next, repeatedly merge (in any order) any two parts $A,B$ of the current partition, as long as $A\cup B$ is contained in some part of $\cal P_{t+1}$.
  Once this is done, the current partition is equal to $\cal P_{t+1}$, and we proceed to the next stage.
  
  During all steps in stage $t$, the set $R$ of resolved pairs is contained in $R_{t+1}$, and $\cal P$ is a coarsening of $\cal P_t$. In particular, for every $r\in\N\cup\set{\infty}$,
  \begin{multline*}
\text{radius-$r$ width of $(\cal P,R)$}\ \le \ 
\text{radius-$r$ width of $(\cal P_t,R_{t+1})$} \\
\le \ \text{radius-$r$ width of the given merge sequence}.
  \end{multline*}
 
  After the last stage, the current partition $\cal P$ has only the single part $V$. 
  Finally, we resolve $V$ with itself. This completes the  construction sequence.
  \end{proof}

\subsection{Construction sequences have linear length}\label{sec:representing}

When designing merge-width based algorithms,
we always assume graphs to be represented by construction sequences, rather than merge sequences,
as the former impose more structure.
A~priori such a construction sequence can be arbitrarily long, as for example, the same pair of parts can be resolved indefinitely.
In this subsection we make an innocuous assumption that guarantees a linear bound on their length.

Namely, call a resolve operation in a construction sequence 
\emph{effective} if it creates a non-zero number of new edges or non-edges. 
That is, an \emph{ineffective} resolve operation 
is a resolve operation performed between a pair $A,B$ of parts 
such that all pairs in $AB$ were already resolved before the operation.
Call a construction sequence \emph{effective} if 
every resolve operation in it is effective.
We assume without loss of generality and without explicitly mentioning it that all construction sequences are effective.
Note that the following lemma also applies if the radius-\(r\) width is bounded, as this implies a bound on the radius-\(1\) width.

\begin{lemma}\label{lem:effective-construction-seq}
  Every effective construction sequence of a graph $G$ of radius-$1$ width $w$ has length at most $(2w+1)|V(G)|$.
\end{lemma}
\begin{proof}
  We rearrange the construction sequence by repeatedly postponing resolve operations.
  Namely, consider a subsequence of consecutive operations in a construction sequence,
  consisting of a sequence of resolve operations, and finishing with the merge
  of two parts $A,B\in\cal P$. We may postpone all resolves which occur in the subsequence, and involve 
  neither $A$ or $B$, until after the merge of $A$ and $B$ is done.
  This result in a valid construction sequence,
  of the same radius-$1$ width and the same length.
  As all resolve operations are effective,
  at most $2w$ resolve operations between pairs $P,Q$ remain in the subsequence: at most $w$ pairs involving $A$, and at most $w$ pairs involving $B$, because each vertex in $A$ (or in $B$) $1$-reaches all parts resolved with $A$ (resp. $B$). 
  In the final construction sequence, at most $2w$ resolve operations are performed per merge,
  followed potentially by a single last resolve on the final part.
  As there can be at most \(|V(G)|-1\) merges, we have at most $(2w+1)|V(G)|$ steps altogether.
  Since the original construction sequence has the same length as the rearanged construction sequence, the result follows.
\end{proof}

%% file: proofs.tex
\section{A locality theorem}\label{sec:logic}
Most of the material in this section is taken from the standalone note~\cite{locality-note}.
We state here the definitions and results from that note that are used in the sequel;
the proofs are deferred to~\cite{locality-note}.
We use a restricted form of the result of \cite{locality-note}, which is sufficient for the purpose of this paper.
The main result is the following locality theorem (see below for definitions).
\begin{restatable}[{\cite{locality-note}} Locality theorem for \distFO]{theorem}{localglobalgame}\label{thm:localglobal}
        Fix $k,q \in \N$. Every \distFO formula $\phi(\bar x)$ of distance rank $(k,q)$
     is equivalent to a boolean combination of
    local \distFO formulas and \distFO scatter sentences, all of distance rank $(k,q)$.    
    This boolean combination can be effectively computed, given $k,q$, and~$\phi$.
\end{restatable}
In \Cref{sec:types} we introduce various notions of \emph{types} and translate \Cref{thm:localglobal} into this framework.
Using this, in \Cref{sec:mc} we show that given a binary structure $A$ and the local type of each vertex in it, we can efficiently compute the new local type of each vertex following a merge operation or a resolve operation.

\subsection{Preliminaries}
 $\N$ denotes the set of non-negative integers. For $n\in\N$, denote $[n]\eqdef\set{1,\ldots,n}$.

We consider finite relational signatures only.
A structure $A$ over the signature $\sigma$,
 or a \emph{$\sigma$-structure}, consists of a (possibly infinite) domain $V(A)$ and the interpretation
  $R^A\subset V(A)^k$ of each relation symbol $R\in \sigma$ of arity $k$.
The \emph{Gaifman graph} of a structure $A$ is the graph with vertices $V(A)$ and edges connecting pairs of distinct vertices which occur together in some tuple of some relation of $A$. By $\dist^A(\cdot,\cdot)$ we denote the shortest path metric  in the Gaifman graph of $A$, and for $r\in\N$ and $u\in V(A)$ denote $N^A_r(u)\coloneqq\setof{v\in V(A)}{\dist^A(u,v)\le r}$.




\subsection{Distance logic and scatter sentences}\label{sec:typedefs}

\paragraph{First-order distance logic.}
We define a logic called \distFO that extends FO by the following \emph{distance atoms}.
\begin{quote}
     For each radius \(r \in \N\), \distFO introduces
        a \emph{binary distance atom} \(\dist(x,y) \le r\) expressing that the distance between \(x\) and \(y\) in the Gaifman graph is at most \(r\).
        We call \(r\) the \emph{radius} of the distance atom.

     For each radius \(r \in \N\) and unary relation symbol \(Y\), \distFO introduces
        a \emph{unary distance atom} \(\dist(x,Y) < r\) expressing that the distance from \(x\) to \(Y\) in the Gaifman graph is smaller than \(r\);
        equivalently, \(\exists y~Y(y) \land \dist(x,y) < r\).
        We call \(r\) the \emph{radius} of the distance atom.
\end{quote}
Formulas of \distFO are built inductively,
starting with usual atomic formulas of first-order logic (relation symbols or equality applied to variables), distance atoms, boolean connectives, and existential quantification (universal quantification is then expressible using negation).

For a \distFO formula $\phi(x_1,\ldots,x_k)$, structure $A$, 
and  tuple $(a_1,\ldots,a_k)\in V(A)^k$
we write $A\models\phi(a_1,\ldots,a_k)$
to denote
that $\phi(a_1,\ldots,a_k)$ holds in $A$, which is defined by induction on the structure of $\phi$, in the expected way.

For a set $\bar x$ of variables, 
we write $\phi(\bar x)$ to indicate that $\bar x$ contains the free variables of the formula $\phi$.
We write \(\dist(\bar x,y) \le r\) as shorthand for \(\bigvee_{x \in \bar x}\dist(x,y) \le r\), and
\(\dist(\bar x,\bar y) \le r\) as shorthand for
\(\bigvee_{x\in \bar x, y \in \bar y}\dist(x,y) \le r\).


\paragraph{Horizon functions.}
Fix the following two \emph{horizon functions} $\rho^-,\rho^+\from \N^2\to \N_{\ge 1}$:
\begin{align*}
   \rho^-(k,q) \eqdef 9^{(k+q+1)q} \quad\quad\quad\text{and}\quad\quad\quad
   \rho^+(k,q) \eqdef 9^{(k+q)(q+1)}.
\end{align*}
(In \cite{locality-note}, more general functions are allowed, but the above choice is sufficient for this paper.)
The functions \(\rho^-\) and \(\rho^+\) will be used to define the allowed radius of distance atoms in \distFO formulas and scatter sentences of a given \emph{distance rank} (defined below), as well as the radius at which \emph{local quantification} is allowed.

\paragraph{Distance rank.}
Let \(k,q \in \N\).
We define \emph{\distFO formulas of distance rank} \((k,q)\) by induction on~$q$,
as formulas with at most $k$ free variables and
quantifier rank at most $q$ which are boolean combinations of:
\begin{itemize}
    \item atoms of first-order logic (relation symbols in $\sigma\cup\set{=}$ applied to variables),
    \item unary and binary distance atoms with radius  \(r\le \rho^-(k,q)\),
    \item if $q\ge 1$, formulas
        \[
            \exists y~\phi(\bar x y)
            \quad\quad\text{or}\quad\quad
            \exists y~\bigl(\dist(\bar x,y) \le \rho^+(k+1,q-1)\bigr) \land \phi(\bar x y)
        \]
        where \(\phi(\bar x y)\) has distance rank \((k+1,q-1)\).
\end{itemize}
A \distFO formula is \emph{local}
if it never uses the \emph{unrestricted quantification} $\exists y~\phi$
(only \emph{local quantification}, that is, of the latter form above, is allowed).

\paragraph{Scatter sentences.}
A set \(S\) of vertices of a structure $A$ is \emph{\(r\)-scattered} if all vertices in \(S\)
have pairwise distance larger than \(r\) in the Gaifman graph of $A$. (Thus, a \(1\)-scattered set is an independent set.)

For every finite structure $A$, radius \(r\), formula \(\beta(x)\), 
define $s_{A,r,\beta}$ as the size of the inclusion-wise maximal $r$-scattered subset 
of \(X\coloneqq\{ a \in V(A) \mid A \models \beta(a)\}\) obtained by the following greedy process.
Assuming \(v_1,\dots,v_i\) have already been chosen, select \(v_{i+1}\) as the smallest element of~$X$ (with respect to some arbitrary, fixed order on \(V(A)\))
 which is of distance larger than \(r\) in the Gaifman graph of $A$ from \(v_1,\dots,v_i\).
If no such \(v_{i+1}\) can be found, terminate and set $s_{A,r,\beta} \coloneqq i$.
Note that the value $s_{A,r,\beta}$ can be efficiently computed,
given $A$, an order on $V(A)$, and $X$.
(In \cite{locality-note}, other choices of the value $s_{A,r,\beta}$ are allowed, but the above choice is sufficient in this paper. For infinite structures $A$, which are not considered in this paper, the definition of $s_{A,r,\beta}$ can be extended in various ways.)

We define logical \emph{scatter sentences}
\(\textbf{scatter}(r,\beta(x),t)\) with \(A \models \textbf{scatter}(r,\beta(x),t)\) if and only if $s_{A,r,\beta}\ge t$.
We define \emph{scatter sentences of distance rank \((k,q)\)}, for \(q\ge 1\),
as all scatter sentences \(\mathbf{scatter}(r,\beta(x),t)\) such that
\(t\le k+q\) and
for some \(i\in\{1,\ldots,q\}\),
\(\beta(x)\) is a local \(\distFO\) formula of distance rank \((k+i,q-i)\), and 
\begin{align}\label{eq:scatter-radius}
4\rho^-(k+i,q-i)\le r\le 9^{k+i}\rho^-(k+i,q-i)\le \rho^-(k,q).
\end{align}

We have defined all the notions occurring in the statement of \Cref{thm:localglobal}, repeated below.
\localglobalgame*

\paragraph{Some facts.}
We repeat some statements from \cite{locality-note}.

\begin{observation}[{\cite[Observation 2.1]{locality-note}}]\label{obs:preservedistrank}
    Every \distFO formula of distance rank \((k+1,q-1)\) with at most \(k\) free variables also has distance rank \((k,q)\).
\end{observation}

\begin{observation}[{\cite[Observation 2.3]{locality-note}}]\label{obs:preservescatterrank}
    Every scatter sentence of distance rank \((k+1,q-1)\) is also a scatter sentence of distance rank \((k,q)\).
\end{observation}

The following lemma states that for local \distFO formulas, the truth value is determined by a bounded-radius neighborhood of the free variables. 

\begin{lemma}[Semantic locality, {\cite[Lemma~2.2]{locality-note}}]\label{lem:ltp-local-subgraph}
    Let \(k,q \in \N\) with $k\ge 1$, and let
    \(\phi(\bar x)\) be a local \distFO formula of distance rank \((k,q)\), and let $r=\rho^-(k,q)$.
    Then  for every structure $A$ and tuple $\bar a\in A^{\bar x}$,
we have that
$$A\models \phi(\bar a) \quad\Longleftrightarrow\quad A[N_r(\bar a)]\models \phi(\bar a).$$
\end{lemma}

\begin{lemma}[Separation lemma, {\cite[Lemma 3.1]{locality-note}}]\label{lem:far-ltp}
    Fix $k,q\in\N$ and let $\phi(\tup x,\tup y)$ be a local \distFO formula of distance rank $(k,q)$.
    The formula
    $$\Big(\dist(\tup x,\tup y)>\rho^-(k,q)\Big) \ \land\ \phi(\tup x,\tup y)$$
    is equivalent to a disjunction of formulas of the form
    $$\Big(\dist(\tup x,\tup y)>\rho^-(k,q)\Big) \ \land\  \alpha(\tup x)\land \beta(\tup y),$$
    where $\alpha$ and $\beta$ are local \distFO formulas of distance rank $(k,q)$.
    Moreover, this disjunction can be effectively computed, given $k,q$, and $\phi$.
\end{lemma}

\subsection{Types}\label{sec:types}
We define two variants of \emph{types} for \distFO,
following the standard notions from model theory.

Fix $k,q\in\N$ and a finite relational signature $\sigma$.
Fix a $\sigma$-structure \(A\) and tuple of vertices \(\bar a\in V(A)^\ell\) with $\ell\le k$.
\begin{itemize}
    \item The  \emph{$(k,q)$-type}
of \(\bar a\) in~\(A\)
denoted \(\tp_{k,q}(A,\bar a)\), is the set of all \distFO formulas \(\phi(x_1,\ldots,x_\ell)\) of distance rank \((k,q)\)
with \(A \models \phi(\bar a)\).
\item
The \emph{local $(k,q)$-type}
of \(\bar a\) in~\(A\)
denoted \(\ltp_{k,q}(A,\bar a)\), is the set of all local \distFO formulas \(\phi(x_1,\ldots,x_\ell)\) of distance rank \((k,q)\)
with \(A \models \phi(\bar a)\).
\item
The \emph{scatter \((k,q)\)-type} of \(A\), denoted \(\stp_{k,q}(A)\) is the set of all scatter sentences of distance rank \((k,q)\) that hold in \(A\).
\end{itemize}
We assume all formulas to be normalized\footnote{This is defined by induction on the quantifier rank of the formula, by saying that a \emph{normalized} formula of quantifier rank $q$ is a boolean combination of normalized formulas of quantifier rank $q-1$ or atomic formulas, where the boolean combination is in disjunctive normal form with no two identical disjuncts and no two identical conjuncts in one disjunct. Every \distFO formula can be effectively rewritten into an equivalent normalized formula, without changing its rank, or its distance rank.}, and thus the size of each type is bounded in terms of \(k\), \(q\), and the signature of the considered structure.
Furthermore, we assume that all those types store additionally the information about $\sigma$, $k$, and $q$
(as well as $\ell$, in the case of $\ltp$ and $\tp$).

We rephrase \Cref{thm:localglobal} as follows.

\begin{theorem}[Reformulation of \Cref{thm:localglobal}]\label{thm:localglobal'}
    Fix $k,\ell,q\in\N$ with $\ell\le k$
    and a finite relational signature $\sigma$.
For every $\sigma$-structure $A$ and tuple $\tup a\in V(A)^\ell$,
the type $\tp_{k,q}(A,\tup a)$ depends (effectively) only 
on the local type $\ltp_{k,q}(A,\tup a)$ and on the scatter type $\stp_{k,q}(A)$.

More precisely, there is a computable function $f$ 
such that for every finite relational signature $\sigma$, for all $k,q,\ell\in\N$ with $\ell\le k$, and for every $\sigma$-structure $A$ and tuple $\tup a\in V(A)^\ell$ we have
$$\tp_{k,q}(A,\tup a)=f(\ltp_{k,q}(A,\tup a), \stp_{k,q}(A)).$$
\end{theorem}
\begin{proof}
Fix $\sigma,k,\ell,q, A,\tup a$ as above.
Let $\tau=\ltp_{k,q}(A, \tup a)$ and $\tau'=\stp_{k,q}(A)$.
We show how to compute $\tp_{k,q}(A,\tup a)$ by an effective procedure, given $\tau,\tau'$, and not  $A$ and $\tup a$.

Note that for every (suitably normalized) local \distFO formula $\psi(x_1,\ldots,x_\ell)$ of distance rank $(k,q)$,
 $$A\models \psi(\bar a)\quad\iff\quad \psi(x_1,\ldots,x_\ell)\in \tau=\ltp_{k,q}(A,\bar a).$$ 
  This can be determined effectively, given $\psi$ and $\tau$.

  Similarly, given a (suitably normalized) scatter sentence $\alpha$, 
  it holds that 
  $$A\models \alpha\quad\iff\quad \alpha\in \tau'= \stp_{k,q}(A),$$ 
  and this can be determined effectively, given $\alpha$ and $\tau'$.

By induction, 
for every boolean combination $\psi(\bar x)$
of local \distFO formulas and scatter sentences of distance rank $(k,q)$,
whether or not $A\models \psi(\bar a)$ holds can be effectively determined basing only $\psi$ and on $\tau$ and $\tau'$ (and not on $A$ and $\bar a$).

    Next, for every \distFO formula $\phi(x_1,\ldots,x_\ell)$ of distance rank $(k,q)$, we can determine 
    whether $A\models\phi(\bar a)$, as follows. 
Apply \Cref{thm:localglobal} to compute an equivalent formula $\psi(x_1,\ldots,x_\ell)$ which is a boolean combination 
of local \distFO formulas and scatter sentences of distance rank $(k,q)$
(each of which can be normalized effectively).
Then test whether $A\models \psi(\bar a)$, as described earlier.
This depends effectively only on $\phi$, $\tau$, and $\tau'$.

Finally, we can compute $\tp_{k,q}(A,\tup a)$ as the set of all (normalized) \distFO formulas $\phi(x_1,\ldots,x_\ell)$ of distance rank $(k,q)$ such that $A\models\phi(\bar a)$.
Since the set of all normalized \distFO formulas $\phi(\bar x)$ of distance rank $(k,q)$ can be effectively computed, and for each of them we can determine whether $A\models\phi(\bar x)$ as described above,
this shows that $\tp_{k,q}(A,\tup a)$ effectively depends only on $\tau$ and $\tau'$.
\end{proof}

We also rephrase \Cref{lem:far-ltp} as follows (we omit the proof which follows similarly as the proof of \Cref{thm:localglobal'}).
\begin{lemma}\label{lem:far-ltp'}
    For every $\sigma$-structure $A$, numbers $k,q\in \N$ and 
    tuples 
    $\tup a\in V(A)^{\tup x}$ and $\tup b\in V(A)^{\tup y}$ with $|\tup x|+|\tup y|\le k$ 
    and $\dist^A(\tup a,\tup b)>\rho^-(k,q)$,
    $\ltp_{k,q}(A,\tup a\tup b)$ depends (effectively) only on $\ltp_{k,q}(A,\tup a)$
    and $\ltp_{k,q}(A,\tup b)$.
\end{lemma}

%% file: mc.tex
\section{Model checking}\label{sec:mc}
In this section, we prove the main model-checking theorem, \Cref{thm:main}.
The proof uses a slightly more general statement for binary structures, \Cref{thm:main-algorithm}.
We proceed to define the notions of construction sequences and merge-width, and of types for binary structures, which are needed to state and prove \Cref{thm:main-algorithm}.
We rely on the locality theorem, the local-subgraph lemma, and the far-apart decomposition from the standalone locality note~\cite{locality-note}.

\subsection{Construction sequences for binary structures}\label{sec:computing}
\newcommand{\merge}[2]{\mathrm{merge}_{#1,#2}}
\newcommand{\resolve}[3]{\mathrm{resolve}^{#1}_{#2,#3}}
In \Cref{sec:mc}, we consider binary relational signatures only.
More precisely, a (binary, relational) \emph{signature} $\sigma$ consists of a set $\sigma_2$ of binary relation symbols, and a set $\sigma_1$ of unary relation symbols.
A \emph{binary structure} is a structure over a binary relational signature.

We first define the notion of construction sequences, and of merge-width, for binary structures.
Throughout \Cref{sec:mc}, fix a signature $\sigma$
consisting of finitely many unary and binary relational symbols \(\sigma_1\) and \(\sigma_2\).
Let $\pi$ consist of infinitely many unary predicates $P_1,P_2,\ldots$ not occurring in $\sigma$,
and let $\sigma^*\coloneqq\sigma\cup\pi$.
We call the predicates in $\pi$ \emph{part predicates}.
A \emph{partitioned $\sigma$-structure} (called partitioned structure for brevity)
is a $\sigma^*$-structure $\str A$ in which
the predicates $P_i\in\pi$ define pairwise disjoint subsets of $\str A$, which cover $V(\str A)$.
In particular, at most $|V(\str A)|$ part predicates are non-empty.

For a partitioned structure $\str A$ and
 part predicates $P,Q\in \pi$ and binary relation $R\in\sigma_2$,
define the following partitioned structures:
\begin{itemize}
\item $\resolve RPQ(\str A)$, defined as the partitioned structure $\str B$ obtained from $\str A$ by setting $$R^\str B\coloneqq R^\str A\cup (P^\str A\times Q^\str A-\bigcup_{S\in\sigma_2} S^\str A),$$ and leaving all other predicates as in $\str A$.

\item $\merge PQ(\str A)$, defined as the partitioned structure $\str B$ obtained from $\str A$
 by setting $$P^{\str B}\coloneqq P^{\str A}\cup Q^{\str A}\text{\quad and \quad} Q^{\str B}\coloneqq\emptyset.$$
\end{itemize}

 For a partitioned $\sigma$-structure $\str A$, $\reach_r^{\str A}(a)\subset \sigma_1\cup\pi$ denotes the set of unary predicates that are reachable from $a$ by a path of length at most $r$ in the Gaifman graph of~$\str A$.
Moreover, the \emph{radius-$r$ width} of a partitioned structure $\str A$
was defined as $$\max_{a\in V(\str A)}\bigl|\reach_r^{\str A}(a)\bigr|.$$

\begin{definition}[Construction sequences]
An \emph{initial} partitioned structure
is a partitioned structure whose Gaifman graph is edgeless,
and all part predicates contain at most one element.

  A \emph{construction sequence} of a structure $\str A$
  is a sequence of partitioned structures $\str A_1,\ldots,\str A_m$
  such that
  \begin{itemize}
    \item $\str A_1$ is an initial partitioned structure,
    \item $\str A_m$ has only one nonempty part predicate,  $\str A$ is a reduct of $\str A_m$,
    and the Gaifman graph of $\str A_m$ is the complete graph on $V(\str A)$,
    and
\item for each $t\in[m-1]$, we have that $\str A_{t+1}=F(\str A_t)$, for
some $F\in\setof {\merge PQ,\resolve RPQ}{P,Q\in\pi,R\in \sigma_2}$.
  \end{itemize}
  The \emph{radius-$r$ width} of the construction sequence is the maximum
  radius-$r$ width of all the structures $\str A_1,\ldots,\str A_m$.
  The \emph{radius-$r$ merge-width} of $\str A$ is the minimum
  radius-$r$ width of a construction sequence of $\str A$.
\end{definition}

For graphs, the definition above is compatible with the construction sequences
from \Cref{sec:intro} in the following sense.

\begin{lemma}\label{lem:construction-construction}
Let \(G=(V,E_G)\) be a graph, and 
let $A_G$ be the corresponding relational structure with domain $V$ and a binary relation $E=\setof{(u,v)\in V^2}{u\neq v, uv\in E_G}$.
Every construction sequence of \(G\) of
length \(\ell\) can be transformed into a construction sequence
\(\str A_1,\ldots,\str A_m\) of $A_G$ of length \(m\le |V|+2\ell+1\), and the same radius-$r$ width, for every \(r\in\N\cup\{\infty\}\).
\end{lemma}
\begin{proof}   
    The construction sequence for $A_G$ 
    uses the signature $\sigma=\set{E,N}$. Starting with the partition of $V$ into singletons, the sequence first resolves applies the operation $\resolve{N}{P}{P}$ for each singleton part $P$. After this sequence of steps we have $N=\setof{(v,v)}{v\in V}$.
    
    Next we proceed with the construction sequence by replicating the steps in the given graph construction sequence. A merge
step is simulated by the corresponding operation \(\merge PQ\). A positive
resolve of two parts \(P,Q\) is simulated by the two structure operations
\(\resolve EPQ\) and \(\resolve EQP\); a negative resolve is simulated in the
same way using \(N\) instead of \(E\). 

By induction over the sequence, the off-diagonal pairs contained in
\(E\cup N\) are exactly the currently resolved vertex pairs of the graph
construction sequence, with both orientations present. Therefore the Gaifman
graph of the current partitioned structure is exactly the graph \((V,R)\) of
currently resolved pairs. Since \(\sigma_G\) has no unary symbols besides the
part predicates, the radius-\(r\) width is preserved for every \(r\).
\end{proof}

\subsection{Main lemma}

The model checking algorithm processes a given construction sequence, step by step.
The following lemma handles a single step of the computation.

\begin{restatable}{lemma}{computestep}\label{lem:local-type-resolve-merge}
    There is an algorithm which takes as input
    a $\sigma$-structure $\str A$, \(\bigl(\ltp_{k,q}(\str A, v):v\in V(\str A)\bigr)\), and
    \begin{itemize}
        \item either symbols \(P,Q \in \sigma_1\) defining a structure \(\str B \eqdef \merge PQ(\str A)\),
        \item or symbols \(P,Q \in \sigma_1\) and \(R \in \sigma_2\) defining a structure \(\str B \eqdef \resolve RPQ(\str A)\),
    \end{itemize}
    and computes
    $\bigl(\ltp_{k,q}(\str B, v):v\in V(\str B)\bigr)$
    in time $O_{k,q,\sigma_2,w}(|V(\str A)|+|E(\str A)|)$,
    where \(w\) is the radius-\((2\rho^-(k-1,q+1)+2)\) width of \(B\).
\end{restatable}

\Cref{lem:local-type-resolve-merge} easily yields the following generalization of \Cref{thm:main} to binary structures.

\begin{restatable}{theorem}{mainalgorithm}\label{thm:main-algorithm}
There is an algorithm which,
given a binary $\sigma$-structure $\str A$ together
 with its construction sequence
 and a first-order $\sigma$-sentence $\phi$,
 determines whether $\str A\models \phi$
in time
 $O_{\sigma,q,w}(|V(\str A)|^2 \cdot m)$,
 where:
    \begin{itemize}
        \item $m$ is the length of the construction sequence,
        \item $q$ is the quantifier rank of $\phi$,
        \item $w$ is the radius-\((2\rho^-(0,q+1)+2)\) width of the construction sequence.
    \end{itemize}
\end{restatable}

\begin{proof}[Proof of \Cref{thm:main-algorithm}]
   Let $q$ be the quantifier rank of $\phi$.
For  $t=1,\ldots,m$ we compute
$$(\ltp_{1,q}(\str A_t,a):a\in V(\str A)),$$
by induction on $t$.
In the base case of $t=1$, the structure $\str A_1$ is an initial partitioned structure.
  Therefore, its Gaifman graph is edgeless and all distance atoms of the dist-FO logic are vacuous.
  In effect, $\ltp_{1,q}(\str A_1,a)$ is completely determined by the
  atomic type $\tau(x)$ of $a$ in $\str A_1$,
as well as the  quantifier-rank $q$ first-order type of $a$ in the structure $\str A'$ obtained from $\str A_1$ by forgetting the part predicates $Q\in\pi$, as well as all binary predicates. As the structure $\str A'$ involves only unary predicates,
evaluating first-order formulas of fixed quantifier rank $q$ can be done in time $O_{q,|\sigma|}(|V(\str A')|)$.
This completes the case $t=1$.

In the inductive step,
apply \Cref{lem:local-type-resolve-merge} to compute $\bigl(\ltp_{1,q}(\str A_t,a):a\in V(\str A)\bigr)$ from $\bigl(\ltp_{1,q}(\str A_{t-1},a):a\in V(\str A)\bigr)$ in time $O_{q,\sigma_2,w}(|V(\str A)|+|E(\str A_{t-1})|)\le O_{\sigma,q,w}(|V(\str A)|^2)$. As there are $m$ computation steps, the total running time is $O_{\sigma,q,w}(|V(\str A)|^2\cdot m)$.

Finally, to determine whether $A\models\phi$,
we view $\phi$ as a formula $\phi(x)$ with one free variable, and check whether $\phi$ is in $\ltp_{1,q}(\str A_m,a)$ for any chosen vertex $a\in V(\str A)$, which can be done in time $O_{q,\sigma}(1)$
 (if $V(A)$ is empty, we evaluate $\phi$ by brute force).
Note that since the Gaifman graph of $\str A_m$ is complete, $\phi$ can be viewed as a local \distFO formula of distance rank $(1,q)$, so the above check is well-defined.
\end{proof}

 \Cref{thm:main-algorithm} easily implies \Cref{thm:main}, as we now argue.
\begin{proof}[Proof of \Cref{thm:main}]
  Let \(q\) be the quantifier rank of \(\phi\), and let $r=(2\rho^-(0,q+1)+2)\le 2^{O(q^2)}$. 
  Let $w$ be the radius-\(r\) width of the input construction sequence of $G$, and 
let \(\ell\) be its length.
  Since we assume the sequence is effective and has radius-\(r\) width \(w\), it also
  has radius-\(1\) width at most \(w\), and \Cref{lem:effective-construction-seq}
  gives \(\ell\le (2w+1)|V(G)|\).

  By \Cref{lem:construction-construction}, the graph construction sequence
  induces a construction sequence of the corresponding binary \(\{E\}\)-structure
  \(\str A_G\) of length \(m\le |V(G)|+2\ell+1\le O_w(|V(G)|)\) and radius-\(r\) width $w$.
  Then
  \[
      G\models\phi \quad\iff\quad \str A_G\models\phi.
  \]
  By \Cref{thm:main-algorithm}, we can decide whether \(A_G\models\phi\) in time
  \[
      O_{q,w}(|V(G)|^2\cdot m)=O_{q,w}(|V(G)|^3).\qedhere
  \]
\end{proof}
The remainder of \Cref{sec:mc} is devoted to the proof of \Cref{lem:local-type-resolve-merge}.

\subsection{Transforming types}
\label{sec:trans-tp}
First,
we show that if $A$ is a structure
and $B$ is obtained from $A$ via either a merge or a resolve operation,
then given a type $\tp_{k,q}(A,\tup a)$ of a tuple $\tup a$ in $A$, we can determine
its type $\tp_{k,q}(B,\tup a)$ in~$B$.
To prove this, it is enough to observe that
every formula $\phi(\tup x)$ of \distFO can be transformed into a formula $\psi(\tup x)$
of the same distance rank, so that
\[B \models \phi(\bar v)\quad\iff\quad A \models \psi(\bar v)\qquad \text{for all }\bar v\in V(B)^{\bar x}.\]
This is based on formula rewriting, using the definitions of the merge and resolve operations.
The starting observation is that if $B$ is obtained from $A$ by such an operation, then the atomic predicates of $B$ correspond to quantifier-free formulas in $A$, respecting the equivalence above.
For plain first-order logic, this is enough to conclude that every first-order formula $\phi(\tup x)$ can be transformed into a formula $\psi(\tup x)$ of the same quantifier rank so that the above equivalence holds.
For \distFO, we also need to analyse how  distance atoms change from $B$ to $A$.
The case of a merge operation is easy, since it does not change distances;
this is handled by the next lemma.

\begin{lemma}\label{lem:merge-interp}
    Let $\str A$ be a $\sigma$-structure and \(\str B \eqdef \merge PQ(\str A)\)
    for some $P,Q\in\sigma_1$.
    Every \distFO formula \(\phi(\bar x)\) can be transformed into a \distFO formula \(\psi(\bar x)\) of the same distance rank
    such that
    \[B \models \phi(\bar v)\quad\iff\quad A \models \psi(\bar v)\qquad \text{for all }\bar v\in V(B)^{\bar x}.\]
    The formula $\psi$ can be computed in time $O_{k,q,\sigma}(1)$,
    given $P, Q$, and $\phi$. Moreover, if $\phi$ is a local \distFO formula, then so is $\psi$.
\end{lemma}
\begin{proof}
    Note that \(A\) and \(B\) have the same Gaifman graph.
    Observe that for all vertices \(u\),
    \[
        P^B(u) ~\iff~ P^A(u) \lor Q^A(u) \quad\quad\text{and}\quad\quad Q^B(u) ~\iff~ \text{false}.
    \]
    More generally, for all $r\in\N$,
    \begin{align*}
        \dist(u, P^B)< r ~&\iff~
        \bigl(\dist(u, P^A)< r\bigr) \lor
        \bigl(\dist(u, Q^A)< r\bigr),
    \\ \dist(u, Q^B)< r~&\iff~ \text{false}.
    \end{align*}

    Thus, we can transform \(\phi(\tup x)\) to \(\psi(\tup x)\), by replacing each \(P\)- or \(Q\)-atom and unary distance atom with the corresponding quantifier-free formula on the right-hand side of these equivalences.
    This obviously preserves  the truth of the formula, the distance rank, and locality.
\end{proof}
\begin{corollary}\label{cor:merge-interp}
     Let $\str A$ be a $\sigma$-structure and \(\str B \eqdef \merge PQ(\str A)\)
    for some $P,Q\in\sigma_1$.
    For every tuple $\tup a$ of at most $k$ elements of $A$,
    the type $\tp_{k,q}(B,\tup a)$
    only depends on $P,Q$, and the type
    $\tp_{k,q}(A,\tup a)$,
    and is computable from these data in time $O_{k,q,\sigma}(1)$.
Similarly, given $P,Q,$ and the local type $\ltp_{k,q}(A,\tup a)$,
the local type $\ltp_{k,q}(B,\tup a)$ can be computed in time $O_{k,q,\sigma}(1)$.
\end{corollary}
\begin{proof}
    Let \(\tup a\) be a tuple of elements of $A$.
    To decide whether a formula $\phi(\tup x)$ is in $\tp_{k,q}(B,\tup a)$,
    it suffices to  transform $\phi(\tup x)$ into a formula $\psi(\tup x)$ as in \Cref{lem:merge-interp},
    and to then check whether $\psi(\tup x)$ is in $\tp_{k,q}(A,\tup a)$.
    This crucially uses the fact that the formula transformation does not increase the distance rank of the considered formulas. Local types are handled the same way, as the transformation preserves locality.
\end{proof}

The next lemma handles the case of
resolves, and those change distances. Crucially, as the Gaifman graph of $B=\resolve RPQ(\str A)$ contains the Gaifman graph of $A$, distances in $B$
cannot become larger than in $A$.
We need to analyse how
this affects distance atoms, as well as
local quantification.
The key observation is that only distances
between pairs of vertices which were close to either $P$ or $Q$ in $A$ are affected.
Binary distance atoms in $B$ are handled thanks to the unary distance atoms of \distFO,
which allow to measure distances to $P$ or to $Q$ in $A$. This is the reason for why unary distance atoms are included in \distFO.
Similarly, we can handle local quantification in $B$. However, following this transformation
requires introducing unrestricted quantification.
Thus, unlike \Cref{lem:merge-interp}, the transformation in \Cref{lem:resolve-interp} does not preserve locality of formulas
(local formulas will be treated separately in \Cref{lem:local-type-resolve-merge}).

\begin{lemma}\label{lem:resolve-interp}
    Let $\str A$ be a $\sigma$-structure and \(\str B \eqdef \resolve RPQ(\str A)\)
    for some $P,Q\in\sigma_1$ and $R\in\sigma_2$.
    Every \distFO formula \(\phi(\bar x)\) can be transformed into a \distFO formula \(\psi(\bar x)\) of the same distance rank
    so that
    \[B \models \phi(\bar v)\quad\iff\quad A \models \psi(\bar v)\qquad \text{for all }\bar v\in V(B)^{\bar x}.\]
    The formula $\psi$ can be computed in time $O_{k,q,\sigma}(1)$,
    given $\phi, P, Q,R$, and the distances  $\dist^A(P^A,Y)$ and $\dist^A(Q^A,Y)$ for $Y\in \sigma_1$.
\end{lemma}



\begin{proof}
We proceed by induction on the structure of $\phi(\bar x)$.
    We analyze all cases of the structural induction. We assume both $P^A$ and $Q^A$ are non-empty, as otherwise there is nothing to prove.

    %

    \paragraph{Atoms of first-order logic.}
    The relation \(R\) is the only atom from the signature $\sigma$ that changes its interpretation between \(A\) and \(B\).
    Observe for all \(u,v\) that
    \[
        R^\str B(u,v) ~\iff~ R^\str A(u,v) \lor \Bigl(P^\str A(u) \land Q^\str A(v) \land \bigwedge_{S\in\sigma_2} \neg S^\str A(u,v)\Bigr).
    \]
    This transformation does not change the distance rank, and proves the statement for first-order atoms.

    \paragraph{Unary distance atoms up to radius \(\rho^-(k,q)\).}
    For all vertices \(u\) and unary relation symbols \(Y\in \sigma_1\),
    any shortest path from $u$ to $Y^A$ in $B$ uses either $0$, $1$, or $2$ edges with one endpoint in $P^A$ and the other one in $Q^A$. Hence we have:

    \begin{multline*}
      \dist^{B}(u,Y) < r ~\iff~\Bigl(\dist^{A}(u,Y) < r\Bigr)\lor {}\\ \Bigl(\dist^{A}(u,P^{A}) < \bigl(r - 1 - \dist^{A}(Q^{A},Y)\bigr) \Bigr)
                            \lor  \Bigl(\dist^{A}(u,Q^{A}) < \bigl(r - 1 - \dist^{A}(P^{A},Y)\bigr) \Bigr)\lor {}\\
                            \Bigl(\dist^{A}(u,P^{A}) < \bigl(r - 2 - \dist^{A}(P^{A},Y)\bigr) \Bigr)
                            \lor  \Bigl(\dist^{A}(u,Q^{A}) < \bigl(r - 2 - \dist^{A}(Q^{A},Y)\bigr) \Bigr).
    \end{multline*}
    The right-hand sides of the inequalities are numerical constants which depend on the values \(\dist^{A}(Q^{A},Y)\) and \(\dist^{A}(P^{A},Y)\). In the algorithmic part of the statement, we assume those values are given.
    Both sides of the equivalence have the same distance rank.

    \paragraph{Binary distance atoms up to radius \(\rho^-(k,q)\).}

    Every minimal-length path between two vertices \(u\) and \(v\) in \(B\)
    uses $0$, $1$, or $2$ edges from \(B\) that are not in \(A\), that is, edges from \(P^A\) to \(Q^A\) or from \(Q^A\) to \(P^A\).
    Hence, every such path between \(u\) and \(v\) in \(B\)
    is either:
    \begin{itemize}
       \item  a path between \(u\) and \(v\) in \(A\),
       \item a path between \(u\) and \(P^A\) in \(A\) followed by an edge from \(P^A\) to \(Q^A\) (which exists in \(B\)) followed by a path between \(Q^A\) and \(v\) in \(A\),
       \item a path between \(u\) and $P^A$ in \(A\) followed by path of length two  from \(P^A\) to \(P^A\) passing through \(Q^A\), followed by a path between \(P^A\) and \(v\) in \(A\), or
       \item one of the above with the roles of \(P^A\) and \(Q^A\) reversed.
    \end{itemize}
For each $r\in \N$, define the following formula (here $\dist(a,X)\le s$ is shorthand for the unary distance predicate $\dist(a,X)<(s+1)$):
\begin{multline*}
\delta^r_{P,Q}(x,y)\ \equiv\\
\bigvee_{\substack{r_1,r_2:\\r_1+1+r_2=r}} \Bigl(\bigl(\dist(x,P) \le r_1\bigr) \land \bigl(\dist(y,Q) \le r_2\bigr)\Bigr)\ \lor\
        \Bigl(\bigl(\dist(x,Q) \le r_1\bigr) \land \bigl(\dist(y,P) \le r_2\bigr)\Bigr)\quad\lor\\
        \bigvee_{\substack{r_1,r_2:\\r_1+2+r_2=r}} \Bigl(\bigl(\dist(x,P) \le r_1\bigr) \land \bigl(\dist(y,P) \le r_2\bigr)\Bigr)\ \lor\
        \Bigl(\bigl(\dist(x,Q) \le r_1\bigr) \land \bigl(\dist(y,Q) \le r_2\bigr)\Bigr).
\end{multline*}
Then, for all vertices \(u,v\) and distances \(r \in \N\),
    \begin{align}\label{eq:dist}
        &B\models \dist(u,v) \le r ~\iff~ A\models \bigl(\dist(u,v) \le r\bigr) \ \lor\   \delta^r_{P,Q}(u,v).
    \end{align}
    The right hand side only uses (unary and binary) distance atoms with radius at most \(r\), and thus has the same distance rank as the left side.

    \paragraph{Existential quantification.}
    If the statement holds for \(\phi(\bar x y)\), then it clearly also holds for \(\exists y~ \phi(\bar x y)\).
    Consider now a formula
   \[\exists y~\bigl(\dist(\bar x,y) \le \rho^+(k+1,q-1)\bigr) \land \phi(\bar x y)\] of distance rank \((k,q)\).
    As binary distance atoms may only use radii up to \(\rho^-(k+1,q-1)\), we
    cannot directly evoke the equivalence \eqref{eq:dist} for this binary distance
    atom without increasing the distance rank.
    However, by induction, our statement transforms the subformula \(\phi(\bar x y)\) of distance rank \((k+1,q-1)\) into a \distFO formula \(\psi(\bar x y)\) of the same distance rank.

    By commuting existential quantification and disjunction, the equivalence \eqref{eq:dist}  implies for any \(\bar v=(v_1,\ldots,v_\ell)\in A^{|\tup x|}\) and $r\eqdef\rho^+(k+1,q-1)$:
    \begin{align*}
        B &\models \exists y~\bigl(\dist(y,\bar v) \le r \bigr) \land \phi(y\bar v) ~\iff~ \\
        A &\models \exists y~ \bigl(\dist(y,\bar v) \le r\bigr)  \land \psi(y\bar v)\ \ \  \lor\ \ \exists y~  {\bigvee_{i\in [\ell]}}   \delta^r_{P,Q}(v_i,y)
        \land \psi(y\bar v).
    \end{align*}
    The first disjunct involves
    local existential quantification and is a \distFO formula of distance rank $(k,q)$.
    The second disjunct involves unrestricted quantification. Note that the formula $\delta^r_{P,Q}(v_i,y)$ is a boolean combination of unary distance atoms
    with radii up to \(r=\rho^+(k+1,q-1)\le \rho^-(k,q)\), and is therefore also a \distFO formula of distance rank $(k,q)$.
Therefore, the right side has distance rank \((k,q)\).
    This completes the inductive proof.
\end{proof}

We get the following corollary,
analogous  to \Cref{cor:merge-interp}.
\begin{corollary}\label{cor:resolve-interp}
     Let $\str A$ be a $\sigma$-structure and \(\str B \eqdef \resolve RPQ(\str A)\)
    for some $P,Q\in\sigma_1$ and $R\in\sigma_2$.
    For every tuple $\tup a$ of at most $k$ elements of $A$,
    the type $\tp_{k,q}(B,\tup a)$
    only depends on $P,Q,R$, the type
    $\tp_{k,q}(A,\tup a)$,
    and the distances  $\dist^A(P^A,Y)$ and $\dist^A(Q^A,Y)$ for $Y\in \sigma_1$.
Moreover, given this data,
$\tp_{k,q}(B,\tup a)$ can be computed in time $O_{k,q,\sigma}(1)$.
\end{corollary}

\subsection{Transforming local types}
The transformations from
\Cref{sec:trans-tp} imply that
if the structure $B$ is obtained from a fixed structure $A$
by a merge or resolve operation,
then the type of each vertex in $B$ only
depends on its type (of the same distance rank) in $A$.
However, this transformation
is not directly useful algorithmically in the context of merge-width, since the type
of a vertex stores information about
all predicates in $\sigma_1$,
and the number of such predicates may be unbounded. In fact, the hidden constant in the running time $O_{k,q,\sigma}(1)$ of the algorithms from \Cref{sec:trans-tp} is non-elementary in $|\sigma_1|$.
In the context of merge-width, local types are significantly more manageable,
since we assume that locally around any given vertex, the number of (non-empty) unary predicates is bounded and therefore, by \Cref{lem:ltp-local-subgraph},
the local type has bounded size.

To compute the local type of a vertex $a$ in the structure $B$,
given its local type in the structure $A$,
we apply
the transformations from \Cref{sec:trans-tp}
to a structure $A^*$ which contains a sufficiently large neighborhood of $a$,
and use \Cref{lem:ltp-local-subgraph}.
However, as we are now given the local type of $a$ in $A^*$ on input, we first compute its (global) type in $A^*$, using the locality theorem (specifically, \Cref{thm:localglobal'}). For that, we first need to compute the scatter type of $A^*$.

\begin{lemma}\label{obs:stpcompute}
    Given $k,q\in\N$, a binary signature $\sigma$,
    and a $\sigma$-structure \(A^*\)
    and, if $q\ge 1$, also \(\ltp_{k+1,q-1}(A^*,a)\) for each \(a \in V(A^*)\),
    one can compute \(\stp_{k,q}(A^*)\) in time \(O_{k,q,\sigma} (|V(A^*)|+|E(A^*)|)\).
\end{lemma}
\begin{proof}
For each relevant radius and formula, we execute the greedy scattered set process discussed in \Cref{sec:typedefs} for up to \(k+q\) steps.
For each vertex added by the greedy process,
the factor \(|V(A^*)|+|E(A^*)|\) in the run time accounts for the search that is necessary to exclude the vertices that are close to~it.
\end{proof}

\begin{lemma}\label{lem:localglobaltypes-uniform}
    There is an algorithm which takes as input
    \(k,q \in \N\),
    a $\sigma$-structure \(A^*\),
    and \(\ltp_{k,q}(A^*,a)\) for each \(a\in V(A^*)\),
    and computes
    \(\tp_{k,q}(A^*,a)\) for each $a \in V(A^*)$, jointly
    in time \(O_{k,q,\sigma}(|V(A^*)|+|E(A^*)|)\).
\end{lemma}
\begin{proof}
    By \Cref{obs:stpcompute} and \Cref{obs:preservedistrank}
    we can efficiently compute \(\stp_{k,q}(A^*)\).
    By \Cref{thm:localglobal'},
    \(\tp_{k,q}(A^*,a)\) can be computed from
    \(\ltp_{k,q}(A^*,a)\) and \(\stp_{k,q}(A^*)\)
    in time $O_{k,q,\sigma}(1)$.
\end{proof}

\subsection{Proof of \Cref{lem:local-type-resolve-merge}}
Finally, we prove \Cref{lem:local-type-resolve-merge}, repeated below, which allows us to compute the local type of a vertex in $B$ from its local type in $A$ when $B$ is obtained from $A$ by a merge or resolve operation.
\computestep*

\begin{proof}
    The case of a merge operation follows  from \Cref{cor:merge-interp} and \Cref{lem:ltp-local-subgraph}. We therefore consider the case of a resolve operation. We can assume that $P^A$ and $Q^A$ are non-empty, as otherwise the resolve operation does not change the structure and the statement is trivial.

%
%
%
%
    Let \(r\eqdef\rho^-(k-1,q+1)\) and \(S \eqdef N_{2r}^A(P^A \cup Q^A)\).
    Let \(A^*\) and \(B^*\) be the reducts of \(A\) and \(B\) that remove all unary predicates that contain no element from \(S\).
    In the Gaifman graph of \(B\), as \(P\) and \(Q\) are fully connected (that is, form a biclique or clique), the set \(S\) has radius at most \(2r+2\).
    As \(w\) bounds the radius-\((2r+2)\) width of \(B\), we know that \(B^*\) has at most \(w\) unary predicates.
    Hence, so has \(A^*\).

    As \(A^*\) is a reduct of \(A\), it is trivial to compute \(\ltp_{k,q}(A^*,
    v)\) from \(\ltp_{k,q}(A,v)\) by simply removing all formulas that use unary
    predicates not mentioned in \(A^*\).
    Given that, by \Cref{lem:localglobaltypes-uniform}, we can compute 
    the global types \[\bigl(\tp_{k,q}(A^*,v):v\in V(\str A)\bigr)\]
    in time
    $O_{k,q,\sigma_2,w}(|V(\str A)|+|E(A)|)$.
    Since \(B^*=\resolve RPQ(\str A^*)\), by \Cref{cor:resolve-interp}, we can then compute \[(\tp_{k,q}(B^*,v) : v \in V(B)).\] This
    takes time $O_{k,q,\sigma_2,w}(|V(\str A)|+|E(A)|)$,
    as the distances
 $\dist^A(P^A,Y)$ and
    $\dist^A(Q^A,Y)$ for each
    unary predicate $Y$ of $A^*$,
    required by \Cref{cor:resolve-interp},
     can be computed by a breadth-first search in time \(O(|V(A)|+|E(A)|)\),
     and there are at most $w$ such predicates $Y$.
     Given \((\tp_{k,q}(B^*,v) : v \in V(B))\),  in particular we can obtain \((\ltp_{k,q}(B^*,v) : v \in V(B))\),
     in time $O_{k,q}(|V(A)|)$.

    Thus, to prove the lemma, it remains to show that for all \(v \in V(A)\),
    \[
        \ltp_{k,q}(\str B,v)=\begin{cases}
            \ltp_{k,q}(A,v)&\text{if }\dist^{A}(v,P^A \cup Q^A) > r,\\
            \ltp_{k,q}(B^*,v)&\text{otherwise}.
        \end{cases}
    \]
    To show this, first consider \(v \in V(A)\) with \(\dist^A(v,P^A \cup Q^A) > r\).
    Then \(B[N_r(v)]=A[N_r(v)]\), and thus trivially, \[\ltp_{k,q}(B[N_r(v)],v)=\ltp_{k,q}(A[N_r(v)],v).\]
    Hence, by \Cref{lem:ltp-local-subgraph}, also \[\ltp_{k,q}(B,v) = \ltp_{k,q}(A,v).\]
    Next consider \(v \in V(A)\) with \(\dist^A(v,P^A \cup Q^A) \le r\).
    Then \(N_r^B(v) \subseteq S\), and thus by twice applying \Cref{lem:ltp-local-subgraph},
    \[\ltp_{k,q}(B,v) = \ltp_{k,q}(B[N_{r}(v)],v) = \ltp_{k,q}(B^*,v).\qedhere\]
\end{proof}

\section{Closure under interpretations}\label{sec:closure}
We prove \Cref{thm:interp}, repeated below.
\introinterp*

In fact, we prove a stronger result, stated below, in which the input class $\CC$ is a class of structures over a binary signature $\sigma$.
For a \(\sigma\)-structure \(\str A\) and a first-order \(\sigma\)-formula
\(\phi(x,y)\), define \(\phi(\str A)\) as the graph with vertices
\(V(\str A)\) and edges \(uv\) such that
\(\str A\models \phi(u,v)\lor\phi(v,u)\). Our goal is to bound the merge-width of $\phi(\str A)$ in terms of the merge-width of $\str A$. To this end, it will be convenient to exhibit a merge sequence (see \Cref{sec:merge-sequence}), rather than a construction sequence.

\begin{lemma}\label{lem:interp}
  Fix a finite binary signature \(\sigma\), a first-order \(\sigma\)-formula
  \(\phi(x,y)\) of quantifier rank \(q\), and \(r\in\N\).
  Set
  \[
      h\eqdef \rho^-(2,q)
      \quad\text{and}\quad
      s\eqdef \rho^-(1,q+1).
  \]
  If a \(\sigma\)-structure \(\str A\) has a construction sequence of radius \(rh+s\)
  width at most \(w\), then the graph \(\phi(\str A)\) has a merge sequence of
  radius-\(r\) width bounded in terms of \(r,q,w\), and \(\sigma\).
\end{lemma}

\begin{proof}
  Let \(\str A_1,\ldots,\str A_m\) be a construction sequence for \(\str A\)
  of radius-\((rh+s)\) width at most \(w\).
  We may assume that \(\str A\) is a reduct of \(\str A_m\).
  Denote \(V\coloneqq V(\str A)\).

  For \(t\in[m]\), let \(\cal P_t\) be the partition of \(V\) in which two vertices
  \(a,b\) lie in the same part if and only if
  \[
      \ltp_{2,q}(\str A_t,a)=\ltp_{2,q}(\str A_t,b).
  \]
  Define
  \[
      R_t\coloneqq
      \setof{ab\in \binom V2}{\dist^{\str A_t}(a,b)\le h}.
  \]
Then $R_1\subseteq R_2\dots\subseteq R_m$, as $\dist^{A_{i+1}}(a,b)\le \dist^{A_{i}}(a,b)$ for $i\in[m-1]$ and $a,b\in V(\str A)$.

  We prove that
  \begin{align}\label{eq:new-merege-seq}
      (\cal P_1,R_1),\ldots,(\cal P_m,R_m),(\{V\},\binom V2)
  \end{align}
  is a merge sequence of \(\phi(\str A)\) of radius-\(r\) width bounded in terms of
  \(r,q,w,\sigma\).

  First, \(\cal P_1\) is the partition into singletons because in the initial
  partitioned structure each vertex is the unique element of its part predicate.
  Moreover, \(\cal P_t\) coarsens \(\cal P_{t-1}\).  Indeed, the proof of
  \Cref{lem:local-type-resolve-merge}, and the merge case in \Cref{cor:merge-interp},
  show that \(\ltp_{2,q}(\str A_t,a)\) is determined by
  \(\ltp_{2,q}(\str A_{t-1},a)\) and the operation producing \(\str A_t\).

  We next bound the radius-\(r\) width. Fix \(t>1\) and \(v\in V\).
  If \(u\) is at distance at most \(r\) from \(v\) in \((V,R_t)\), then
  \(\dist^{\str A_t}(u,v)\le rh\).
  By \Cref{lem:ltp-local-subgraph}, every local formula contributing to
  \(\ltp_{2,q}(\str A_{t-1},u)\) is evaluated inside the radius-\(s\)
  neighborhood of \(u\) in \(\str A_{t-1}\).  Since the Gaifman graph of
  \(\str A_{t-1}\) is contained in that of \(\str A_t\), all unary predicates
  relevant to this local type are contained in
  \(\reach^{\str A_t}_{rh+s}(v)\), apart from the two part predicates affected
  when the step \(\str A_{t-1}\to\str A_t\) is a merge.  Thus only \(w+2\)
  unary predicates can occur in these local neighborhoods.
  As the number of normalized local \((2,q)\)-types over a signature with at most
  \(w+2\) unary predicates and fixed binary part \(\sigma_2\) is bounded in terms
  of \(q,w,\sigma\), only boundedly many parts of \(\cal P_{t-1}\) are reachable
  from \(v\) in \((V,R_t)\).  The final pair \((\{V\},\binom V2)\) has bounded
  width as well, since the number of parts of \(\cal P_m\) is bounded by the same
  local-type bound.

  It remains to verify the default condition for unresolved pairs.  Fix \(t\in[m]\)
  and parts \(X,Y\in\cal P_t\).  Let \(a,a'\in X\) and \(b,b'\in Y\) be such that
  \(ab,a'b'\notin R_t\).  Then both pairs are at distance larger than \(h\) in
  \(\str A_t\).  By \Cref{lem:far-ltp'} the local type
  \(\ltp_{2,q}(\str A_t,ab)\) is determined by the one-vertex local types
  \(\ltp_{2,q}(\str A_t,a)\) and \(\ltp_{2,q}(\str A_t,b)\); hence
  \[
      \ltp_{2,q}(\str A_t,ab)=\ltp_{2,q}(\str A_t,a'b').
  \]
  By \Cref{thm:localglobal'}, we have that
  \[
      \tp_{2,q}(\str A_t,ab)=\tp_{2,q}(\str A_t,a'b').
  \]
  Repeatedly applying \Cref{cor:merge-interp,cor:resolve-interp} along the fixed
  suffix \(\str A_t,\ldots,\str A_m\) shows that
  \[
      \tp_{2,q}(\str A_m,ab)=\tp_{2,q}(\str A_m,a'b').
  \]
  Since \(\phi(x,y)\lor\phi(y,x)\) is a first-order formula of quantifier rank
  at most \(q\) over the reduct \(\str A\), it follows that either all pairs in
  \(XY-R_t\) are edges of \(\phi(\str A)\), or none of them are.

  Thus \eqref{eq:new-merege-seq} is a merge sequence of \(\phi(\str A)\) with the
  required radius-\(r\) width bound.
\end{proof}

\begin{proof}[Proof of \Cref{thm:interp}]
    Represent each graph \(G\in\CC\) as the binary
  structure \(\str A_G\) with a binary relation representing adjacency in $G$.  
  Note that $\phi(G)=\phi(A_G)$.
  By \Cref{lem:construction-construction},
  the class of binary structures $\setof{A_G}{G\in\CC}$ has bounded merge-width.  
  By 
 \Cref{lem:interp},
 the class $\setof{\phi(A_G)}{G\in\CC}$ has bounded merge-width.
\end{proof}

%% file: cases.tex
\section{Case studies}\label{sec:cases}
We now review the relationship between
merge-width and previously studied graph parameters.
First, it is not difficult to check that 
$\mw_\infty(G)$ is functionally equivalent to the clique-width of $G$. 
\begin{theorem}\label{thm:cw}
  A graph class $\CC$ has bounded clique-width if and only if $\mw_\infty(\CC)<\infty$.
\end{theorem}

    It is not difficult and quite instructive to prove the result directly,
    by translating between construction sequences and clique-width expressions, and vice-versa.
    Such a proof would provide explicit (linear) bounds in each direction.
    To be more succinct, we derive the statement without explicit bounds, by deriving it from results proved later in \Cref{sec:cases}.

\begin{proof}[Proof sketch]

    For the left-to-right implication,
    let $G$ be a graph of clique-width $k$.
    By \cite[Thm. 1]{tww6}, $G$ has a contraction sequence $\cal P_1,\ldots,\cal P_n$, such that at every 
    step $i\in\set{1,\ldots,n}$ of the sequence,
    every connected component of the red graph of $\cal P_i$ contains 
    at most $k'$ parts, for some constant $k'\in\N$ depending on $k$.
    By transforming the contraction sequence into a merge sequence as in the proof of \Cref{lem:tww} below, by \eqref{eq:tww-mw} we derive that the resulting construction sequence has radius-$\infty$ width at most $k'$. Thus, $\mw_\infty(G)\le k'$.
    This proves the left-to-right implication.

    Conversely, suppose that $\mw_\infty(\CC)<\infty$.
    It follows from \Cref{thm:fw-cases} below that $\fw_\infty(\CC)<\infty$.
    Therefore, $\CC$ has bounded clique-width, by \cite[Thm. II.6]{flip-width}.
\end{proof}





\noindent Thus, in a sense, the finite-radius merge-width parameters are local variants of clique-width.



\subsection{Merge-width and twin-width}\label{sec:tww}
We first discuss the relationship of merge-width and twin-width,
whose definition we recall now.
Recall that a \emph{contraction sequence} of a graph $G$ is a sequence of merge operations, which starts with the partition of $V(G)$ into singletons, and ends with the partition with one part.
Two sets $A,B\subset V(G)$ of vertices are \emph{homogeneous} if $AB\subset E(G)$ or $AB\cap E(G)=\emptyset$.
The \emph{red graph} of a partition $\cal P$
is the graph with vertices $\cal P$ and edges connecting pairs $\set{A,B}\in{\cal P\choose 2}$ 
which are not homogeneous. The twin-width of a graph $G$ is the minimum number $d$ 
such that $G$ admits a contraction sequence such that at every step, the red graph of the current partition has maximum degree at most $d$.

\twwintro*
\Cref{thm:tww} follows immediately from the next lemma.
\begin{lemma}\label{lem:tww}
  Fox every $r,d\in \N$ and graph $G$ of twin-width $d$, we have 
  $$\mw_r(G)\le 2+d+\ldots+d^r.$$
\end{lemma}
\begin{proof}
Fix a contraction sequence $\cal P_1,\ldots,\cal P_n$ of $G$, such that
the red graph of each partition $\cal P_t$ has maximum degree at most $d$.
For a vertex \(v\), denote by \(\cal P_t(v)\) the part of \(\cal P_t\) containing \(v\).
For $t\in[n]$, let $R_t$ consist of those pairs $ab\in {V\choose 2}$ such that $\cal P_t(a),\cal P_t(b)$ are either equal, or not homogeneous in $G$.
Then 
  $(\cal P_1,R_1),\ldots,(\cal P_n,R_n),$
  is a merge sequence of $G$. We bound its radius-$r$ width, for fixed $r\in\N$.

  For every $t\in[n]$, part $A\in\cal P_t$ and vertex $a\in A$, 
  we have that 
  \begin{align}\label{eq:tww-mw}
  B^r_{R_t}(a)\subset \bigcup B^r_{\cal P_t}(A),  
  \end{align}
  where $B^r_{R_t}(a)$ is the ball of radius $r$ around $a$ in the graph $(V,R_t)$, 
  and $B^r_{\cal P_t}(A)$ is the ball of radius $r$ around $A$ in the red graph of $\cal P_t$.
  In particular, $B^r_{R_t}(a)$ intersects at most $|B^r_{\cal P_t}(A)|$ parts of $\cal P_t$.
  As degree in the red graph is at most \(d\), the radius-$r$ width of $(\cal P_t,R_t)$ is at most $|B^r_{\cal P_t}(A)|\le 1+d+\ldots+d^r$.
  Since $\cal P_{t-1}$ is a refinement of $\cal P_t$ with exactly one more part,
  it follows that 
  the radius-$r$ width of $(\cal P_{t-1},R_t)$ is at most $2+d+\ldots+d^r$.
Consequently,
$\mw_r(G)\le 2+d+\ldots+d^r.$
\end{proof}


\subsection{Merge-width and Sparsity}\label{sec:sparsity}
The converse to \Cref{thm:tww} is false. 
For instance,  the class 
of graphs of maximum degree bounded by $d$ (a class which has unbounded twin-width already for $d=3$) 
has bounded merge-width.
Indeed, there is a trivial merge sequence for such a graph $G$: 
 $$\bigl(\textit{partition into singletons},\emptyset\bigr),\quad \bigl(\textit{partition with one part},E(G)\bigr).$$
The radius-$r$ width of this sequence is 
the maximum size of a radius-$r$ ball in $G$,
which is at most $1+d+\cdots+d^r$, 
so 
$\mw_r(G)\le 1+d+\cdots+d^r$.


If the above merge sequence for graphs of bounded degree seems uninsightful,
this reflects the fact that graphs of bounded degree are trivial from the perspective of Sparsity theory.
The next result is slightly more 
illuminating.
The proof is already given in \Cref{ex:degeneracy}.


\begin{theorem}\label{thm:deg}
  If $G$ is a $d$-degenerate graph then 
  $\mw_1(G)\le d+2.$
\end{theorem}

Next, we prove:
\beintro*

The simple construction from \Cref{thm:deg} does not trivially generalize to bound the radius-$r$ merge-width 
of graphs in classes of bounded expansion,
as it appears that to bound $\mw_r(G)$ for $r>1$, a more involved construction 
is needed, which we now present.

\medskip

We first recall the definition of classes of bounded expansion for completeness only, as we will not use the original definition, only the characterization given in the upcoming \Cref{fact:be}.
By definition,  a graph class $\CC$ has \emph{bounded expansion} if and only if for every $r\in\N$,
there is some $k\in\R$ such that for every graph $G$
which can be obtained from some graph in $\CC$ by first removing vertices and edges, and then contracting 
some vertex subsets of radius ${\le r}$ to single vertices, 
we have that $|E(G)|\le k\cdot |V(G)|$.
Classes of bounded expansion include every class of bounded maximum degree,
the class of planar graphs, and every graph class that excludes some graph  as a  minor, or as a topological minor.


Fix a graph $G$, number $r\in\N$, and total order $\le$ on $V(G)$. 
A vertex \(u\) is \emph{weakly \(r\)-reachable} from a vertex \(v\) if there is a path of length at most \(r\) from \(v\) 
to \(u\) and all vertices on that path are at least as large as \(u\).
Let \(\wreach_r(G,\le,v)\) be the set of vertices that are weakly \(r\)-reachable from \(v\) with respect to the ordering.
Then the weak \(r\)-coloring number of \(G\) is defined as
\[
\wcol_r(G) \coloneqq \min_{\le} \max_{v \in V(G)} |\wreach_r(G,\le,v)|,
\]
where the minimum ranges over all total orders $\le$ of $V$.

\begin{lemma}[\cite{zhu2009colouring}]\label{fact:be}
A graph class $\CC$ has bounded expansion if and only if $\wcol_r(\CC)<\infty$, for all~${r\in\N}$.
\end{lemma}

\Cref{thm:be} follows immediately from the next lemma,
which 
shows that the 
 radius-$r$ merge-width 
is bounded in terms of the weak $(r+1)$-coloring number.

\begin{lemma}\label{lem:wcol}
  For every $k,r\in\N$ and graph $G$ with $\wcol_{r+1}(G)\le k$, we have
  $$\mw_r(G)\le 3\cdot 2^{k} .$$
\end{lemma}
In particular, we show that $\mw_2(G)$  is bounded in terms of $\wcol_3(G)$,
and we do not know whether it is bounded in terms of $\wcol_2(G)$.

\begin{proof}Fix $r\in\N$, and a graph $G$ with vertex set $V$.
Let $\le$ be an ordering of the vertex set $V$ of $G$ 
such that for every vertex \(v\), the weak $r$-reachability set $\wreach_{r+1}(G,\le,v)$ has size at most $k$.

For $t=1,\ldots,n$, define the following.
\begin{itemize}
  \item Let $L_t$ comprise the $t$ largest elements of $V$ with respect to $\le$, and let $S_t\coloneqq V-L_t$.
  \item Let  $\cal P_t$ be the partition of $V$ into \emph{atomic types}
  over the set $S_t$. That is, $\cal P_t$ partitions $S_t$ into singletons, and partitions vertices  $v\in L_t$ according to their neighborhood $N(v)\cap S_t$~in~$S_t$.
  \item Let $R_t \coloneqq\setof{ab\in E(G)}{a,b\in L_t}$ be the set of all edges in $G$ with both endpoints in $L_t$.
\end{itemize}
We verify that $(\cal P_1,R_1),\ldots,(\cal P_n,R_n)$ is a merge sequence of $G$. Clearly, $\cal P_1$ is the partition into singletons, 
and $\cal P_n$ has one part.
Observe that for $t\in[n]$ and any two parts $A,B$ of $\cal P_t$,
either $A,B\subset L_t$, and therefore  $E(G)\cap AB\subset R_t$,
or, otherwise, at least one of $A,B$ is contained in $S_t$, and then $A$ and $B$ are homogeneous,
as both $A$ and  $B$ are atomic types over $S_t$. 
In any case, $AB-R_t\subset E(G)$ or 
$AB-R_t\subset {V\choose 2}-E(G)$, so the conditions of a merge sequence are satisfied.

Fix $t\in[n]$. Fix two vertices $v,w\in L_t$, such that 
there is a path of length at most $r$ in $(V,R_t)$ from $v$ to $w$.
Then there is a path $\pi$ of length at most \(r\) from $v$ to $w$ in $G$
contained in $L_t$.
In particular, for every vertex $u\in N_G(w)\cap S_t$ 
there is a path from $v$ to $u$ of length at most $r+1$ in \(G\) (namely the path $\pi$ followed by the edge $wu$) such that every inner vertex on that path is in $L_t$.
In other words, $N_G(w)\cap S_t \subseteq \wreach_{r+1}(G,\le,v)$.
 As $|\wreach_{r+1}(G,\le,v)|\le k$ by assumption, and 
 the atomic type of $w$ over $S_t$ is uniquely 
determined by $N(w)\cap S_t$, it follows that 
there are at most $2^k$ parts of $\cal P_t$ 
that are reachable by a path of length at most $r$ from $v$.
Thus, the radius-$r$ width of $(\cal P_t,R_t)$ is at most
$2^k$.

To bound the radius-$r$ width of $(\cal P_{t-1},R_t)$, for $t>1$, notice that every part in $\cal P_{t}$ 
is a union of at most three parts in $\cal P_{t-1}$.
 Namely, if $v$ is the unique vertex with $v\in S_{t-1}-S_{t}$, then the atomic type $A$ of a vertex $u$ over $S_{t}$ 
 is uniquely determined by the atomic type of $u$ over $S_{t-1}$ and the atomic type of $u$ over $\set v$. 
 There are three 
 possible atomic types over $\set v$, corresponding 
 to
 whether $u=v$, or $(u\neq v)\land E(u,v)$, or 
$(u\neq v)\land \neg E(u,v)$. This proves that $A$ is a union of at most three parts in $\cal P_{t-1}$.

It follows that the radius-$r$ width of $(\cal P_{t-1},R_t)$ is at most $3\cdot 2^k$, and we have constructed a merge sequence of $G$ 
of radius-$r$ width at most $3\cdot 2^k$.
\end{proof}

\subsection{Merge-width and flip-width}\label{sec:fw}
We study the relationship between merge-width and flip-width, whose definition we now recall.

For a graph $G$ and 
two sets $A,B\subset V(G)$,
\emph{flipping} the pair $(A,B)$ in $G$ 
results in the graph $G'$ 
with edges $E(G')=E(G)\triangle AB$, where \(\triangle\) denotes the symmetric difference.
Given a partition $\cal P$ of $V(G)$, 
a \emph{$\cal P$-flip} of $G$ is a graph $G'$
obtained from $G$ by flipping arbitrary pairs $A,B\in\cal P$ (possibly with $A=B$).
Thus, a graph $G'$ is a $\cal P$-flip of $G$ if and only if for every pair $A,B\in\cal P$, 
either $E(G')\cap AB=E(G)\cap AB$ or $E(G')\cap AB=AB-E(G)$.
A \(k\)-flip is a \(\cal P\)-flip with \(|\cal P| \le k\).

We recall the definition of flip-width from \cite{flip-width}.
Fix $k,r\in\N$.
The \emph{flip-width game} of radius $r$ and width $k$ is played on a graph $G$ between two players, flipper and runner.
Initially, runner picks a vertex $v_0\in V(G)$.
In round $t=1,2,\ldots$,
    flipper announces a $k$-flip $G_t$ of $G$;
    then runner picks $v_t\in B^r_{G_{t-1}}(v_{t-1})$ -- that is, they traverse 
    a path of length at most $r$ in the \emph{previous} graph $G_{t-1}$.
    Flipper wins if $v_t$ is isolated in $G_t$.
The \emph{radius-$r$ flip-width} of $G$, denoted $\fw_r(G)$,
is the least number $k$ such that flipper has a winning strategy for the flip-width game of width $k$ and radius $r$ on $G$.
We prove:
\begin{theorem}\label{thm:fw-cases}
  Every class of bounded merge-width has bounded flip-width.
  More precisely, for all $r\in\N$ and every graph $G$,
  $$\fw_r(G)\le 4^{\mw_{2r-1}(G)}.$$
\end{theorem}

\Cref{thm:fw-cases} allows to deduce several results about classes of bounded merge-width
from the corresponding results for classes of bounded flip-width.
For instance, we obtain the two corollaries stated in the introduction:
\corbe*
The forward implication in \Cref{cor:be} follows from \Cref{thm:be} and the trivial fact that each class of bounded expansion excludes some biclique as a subgraph. The backwards direction follows from \Cref{thm:fw-cases} and from the analogous statement for classes of bounded flip-width, proved in \cite[Thm. VI.3]{flip-width}.

\cortww*
In \Cref{cor:tww-mw}, we use the notion of merge-width for binary structures, specifically for \emph{ordered graphs} --  graphs equipped with a total order on the vertex set. This is defined in \Cref{sec:computing}.  \Cref{thm:fw-cases} applies to classes binary structures, allowing to derive \Cref{cor:tww-mw} by a similar reasoning as for \Cref{cor:be} above:
The forward direction uses \Cref{thm:tww}and the fact that every graph can be extended with a total order, without increasing its twin-width \cite{tww4}, and the backwards direction uses \Cref{thm:fw-cases} and the analogue of \Cref{cor:tww-mw} proved in \cite[Thm. VII.3]{flip-width}.
For simplicity, below we present the proof of \Cref{thm:fw-cases} only for graph classes. The more general case is analogous, 
but the requires notion of flips and of flip-width for binary structures as defined in \cite[Sec. V.B]{flip-width}.

We state another corollary that follows from \Cref{thm:deg} and an analogous result from \cite[Thm. VI.1]{flip-width}.
\begin{corollary}\label{cor:deg}
  Let $\CC$ be a graph class. Then $\CC$ has bounded degeneracy if and only if $\mw_1(\CC)<\infty$ and $\CC$ excludes some biclique $K_{t,t}$ as a subgraph.
\end{corollary}

\medskip
To prove \Cref{thm:fw-cases},
we start with a simple observation that allows us to rephrase the main condition from the definition 
of a merge sequence in terms of flips, as follows.

\begin{lemma}\label{lem:flip}
    Let $G=(V,E)$ be a graph, let $\cal P$ be a partition of $V$, and let $R\subset {V\choose 2}$.
    The following conditions are equivalent:
    \begin{enumerate}
        \item for all $A,B\in\cal P$, either $AB-R\subset E$, or $AB-R\subset AB - E$,
        \item  there is a $\cal P$-flip $G'$ of $G$ with $E(G')\subset R$.
    \end{enumerate}
    \end{lemma}
    \begin{proof}
        (1$\rightarrow$2).
        The $\cal P$-flip $G'$
        is obtained from $G$ by flipping each pair $A,B\in\cal P$ such that $AB-R\subset E$. As a result $AB-R\subset AB-E(G')$, 
        and therefore $E(G')\cap AB\subset R$ for each $A,B\in\cal P_t$.
        Altogether, $E(G')\subset R$.
    
        (2$\rightarrow$1). Fix $A,B\in\cal P$. We have that
        $E(G')\cap AB\subset R\cap AB$,
        so $AB-R\subset AB-E(G')$.
        As $G'$ is a $\cal P$-flip of $G$,
        we have that either $E(G')\cap AB=E(G)\cap AB$ or $E(G')\cap AB=AB-E(G)$.
        Altogether, either $AB-R\subset E(G)\cap AB$, or  $AB-R\subset AB-E(G)$.
    \end{proof}
    
Using \Cref{lem:flip}, we may therefore redefine merge-width in terms of flips, as follows.



\begin{definition}Let $G$ be a graph.
A \emph{restrained flip sequence} for $G$ is a sequence
\begin{align}\label{eq:mono-fw}
    (\cal P_1,R_1,G_1),\ldots,(\cal P_m,R_m,G_m)    
\end{align}
such that:
\begin{itemize}
    \item $\cal P_1,\ldots,\cal P_m$ is a refining sequence of partitions, starting with the  partition $\cal P_1$ of $V(G)$ with one part, and ending with the partition $\cal P_m$ into singletons,
    \item the graph $G_t$ is a $\cal P_{t}$-flip of $G$, for $t\in [m]$,
    \item ${V\choose 2}= R_1\supseteq \cdots \supseteq R_m=\emptyset$, and
    \item  $E(G_t)\subset R_t$ for $t\in[m]$.
\end{itemize}
The radius-$r$ width of the restrained flip sequence is 
the maximum, over $t\in[m-1]$,
of the radius-$r$ width of $(\cal P_{t+1},R_{t})$.
\end{definition}

Note that in the reformulation above,
the partitions $\cal P_1,\ldots,\cal P_m$ are becoming finer with each step, instead of becoming coarser as in merge sequences. In each step of a restrained flip sequence, we provide a $\cal P_t$-flip $G_t$.
The sequence $R_1\supseteq \cdots\supseteq R_m$
is a descending sequence of subsets of $V\choose 2$ with $R_m=\emptyset$. The set $R_t$ is called the \emph{restraint} at time $t$, as we require that $E(G_s)\subset R_t$ for all $s\in[m]$ with $s\ge t$, thus restraining all the future graphs $G_t,G_{t+1},\ldots,G_m$.  Note that  $E(G_m)\subset R_m=\emptyset$, so $G_m$ is edgeless.

As $G_t$ is a $\cal P_t$-flip of $G$, 
$G_{t-1}$ is a $\cal P_{t-1}$-flip of $G$, and $\cal P_{t-1}$ is coarser than $\cal P_t$,
we may equivalently require that $G_t$ as a $\cal P_{t}$-flip of $G_{t-1}$, rather than of $G$.
Therefore, in the sequence $G_1,G_2,G_3,\ldots$ 
each subsequent graph is a flip of the previous graph.
This is similar to the setting of flipper games considered in \cite{flippers}.


\begin{lemma}\label{lem:mw-flip seq}
    For every graph $G$,
    there is a correspondence between merge sequences and restrained flip sequences for $G$,
    which preserves the length, and the radius-$r$ width of the sequences, for each $r\in\N\cup\set{\infty}$.
\end{lemma}

\begin{proof}
    Fix a merge sequence 
    $(\cal P_1,R_1),\ldots,(\cal P_m,R_m)$
    of $G$.
 By \Cref{lem:flip} (1$\rightarrow$2),
 for each $t\in[m]$,
    there is a  $\cal P_t$-flip $G_t$ of $G$ such that 
$E(G_t)\subset R_t$.
    The sequence 
    $(\cal P_m,R_m,G_m),\ldots,(\cal P_1,R_1,G_1)$
    is a restrained flip sequence for $G$.
    
    Conversely,
    given a restrained flip sequence for $G$ of the form $(\cal P_1,R_1,G_1),\ldots,(\cal P_m,R_m,G_m)$,
    by \Cref{lem:flip} (2$\rightarrow$1) we conclude that 
    the sequence $(\cal P_m,R_m),\ldots,(\cal P_1,R_1)$ is a merge sequence of $G$.
\end{proof}

The next lemma shows that radius-$r$ flip-width is bounded in terms of radius-$(2r-1)$ merge-width, thus proving \Cref{thm:fw-cases}.

\begin{lemma}\label{lem:fw}
    Fix $r,s\in\N$. 
    For every graph $G$, if $\mw_{2r-1}(G)\le s$ then  $\fw_r(G)\le 4^s$. 
\end{lemma}

\begin{proof}
    Let $G$ be a graph with $\mw_{2r-1}(G)\le s$. Then $G$ has a merge sequence of radius-$(2r-1)$ merge-width at most $s$.
    By \Cref{lem:mw-flip seq}, $G$ has a 
 restrained flip sequence of radius-$(2r-1)$ width at most \(s\) of the form
$$(\cal P_1,R_1,G_1),\ldots,(\cal P_n,R_n,G_n).$$

The following lemma will allow us to construct a strategy for flipper in the flip-width game.

\begin{lemma}\label{lem:strategy}
    Fix $t\in [2,n]$, and 
    let $s$ be the radius-$(2r-1)$ width of $(\cal P_t,R_{t-1})$.
    For every $v\in V(G)$ there is a $4^s$-flip $G_t'$ of $G$ such that:
$$B^r_{G_t'}(w)\subset B^r_{G_t}(w)\quad\text{for every $w\in B^r_{G_{t-1}}(v)$.}$$
\end{lemma}
Before proving \Cref{lem:strategy}, we first show how it yields a winning strategy
for flipper in the flip-width game of radius $r$
    and width $4^s$.
Flipper's strategy will be to announce graphs $G_1',G_2',\ldots$, ensuring  that following invariant holds after round $t\ge 1$ of the game:
    \begin{align}\label{eq:fw-invariant}
    B^r_{G_{t}'}(v_t)\subset B^r_{G_t}(v_t).    
    \end{align}
    
    In the first round, when $t=1$, flipper announces $G_1'=G_1$ and runner picks a vertex 
    $v_1$,
    and the invariant holds trivially.
    Suppose the invariant is satisfied after round $t-1$ of the game,
    that is, 
    $$B^r_{G_{t-1}'}(v_{t-1})\subset B^r_{G_{t-1}}(v_{t-1}).$$
    Now flipper announces the $4^s$-flip $G_t'$ of $G$
    given by \Cref{lem:strategy} for $v \coloneqq v_{t-1}$.
    Next, runner picks a vertex $v_t\in B^r_{G_{t-1}'}(v_{t-1})$.
In particular, $v_t\in B^r_{G_{t-1}}(v_{t-1})$, so 
 \eqref{eq:fw-invariant} holds by \Cref{lem:strategy}, and the invariant is fulfilled.

    Playing according to this strategy, flipper wins within $n$ rounds,
    since $E(G_n)=\emptyset$, and therefore $v_n$ is isolated in $G_n'$
    by \eqref{eq:fw-invariant}. This proves that $\fw_r(G)\le 4^s$, and thus \Cref{lem:fw}.
\end{proof}

\begin{remark}\label{rem:duration}
    Observe that the number of rounds needed by flipper to win in the flip-width game of radius $r$ and width $4^s$ can be bounded by $|V(G)|$.
    Namely, the proof of \Cref{lem:fw} above shows that 
    the number of rounds is (at most) equal to the length $n$ of a restrained flip sequence 
     of radius-$(2r-1)$ width at most \(s\) for $G$.
     By \Cref{lem:mw-flip seq}, this corresponds to the length of a radius-$(2r-1)$ merge sequence of width $s$. Any merge sequence of $G$ can be converted into one of length $n=|V(G)|$, while preserving its radius-$r$ width for each $r\in\N$, since if for two consecutive pairs $(\cal P_i, R_i),(\cal P_{i+1},R_{i+1})$ in a merge sequence we have $\cal P_i=\cal P_{i+1}$, then $(\cal P_{i+1},R_{i+1})$ can be dropped from the merge sequence.
\end{remark}

We now prove \Cref{lem:strategy}.

    \begin{proof}[Proof of \Cref{lem:strategy}]
        Fix $v\in V(G)$. Let
 $\cal N\subset \cal P_t$ consist of all parts $A\in\cal P_t$ that can be reached by a path of length at most $2r-1$ by from $v$ in the graph $(V,R_{t-1})$.
In particular, $|\cal N|\le s$.

Let $G_t'=G'=(V,E)$ be the graph 
such that for any two parts $A,B\in\cal P_t$:
$$E'\cap AB=
\begin{cases}
    E(G)\cap AB&\text{if $A,B\notin\cal N$}\\
    E(G_t)\cap AB&\text{otherwise}.
\end{cases}$$

\begin{claim}\label{cl:k-flip}
    $G'$ is a $4^s$-flip of $G$.
\end{claim}
\begin{claimproof}
    As $G_t$ is a $\cal P_t$-flip of $G$, so is $G'$.
    Therefore, for any two parts $A,B\in\cal P_t$, 
    the set $E'\cap AB$ is either equal to $E\cap AB$,
    or to $AB-E$.

    For every part $A\in\cal N$, let $U(A)$ be the union of all parts 
    $B\in\cal P_t$ such that $E'\cap AB= AB-E$.
    Then for every $B$ with $B\subset U(A)$ or $B\subset V- U(A)$,
    we have that
    $E'\cap AB=E\cap AB$ or 
     $E'\cap AB=AB-E$.

    Consider the set family:
    $$\cal F\coloneqq \cal N \cup \setof{U(A)}{A\in \cal N}.$$ Then $|\cal F|\le 2|\cal N|\le 2s$.
    
    Let $\cal Q$ be the partition 
    of $V$ such that 
    two vertices $a,b$ are in the same part of $\cal Q$ if and only if $a$ and $b$ belong to the same sets in $\cal F$.
    Then $|\cal Q|\le 2^{|\cal F|}\le 4^{s}$.
    Moreover, for every $A,B\in \cal Q$,
    either $E'\cap AB=E\cap AB$,
    or $E'\cap AB=AB-E$.
    It follows that 
    $G'$ is a $4^s$-flip of $G$.
\end{claimproof}

\begin{claim}\label{cl:balls}
    For all $w\in B^r_{G_{t-1}}(v)$
 and $i=0,\ldots,r$ we have:
    $$B^i_{G'}(w)\subset B^i_{G_t}(w).$$
\end{claim}

\begin{claimproof}
    We induct on $i$. For $i=0$ the statement is trivial.
In the inductive step, fix $1\le i\le r$ and $u\in B^i_{G'}(w)$ with $u\neq w$.
Then there is some  $a\in B^{i-1}_{G'}(w)$ with $ua\in E(G')$.
By inductive assumption, $a\in B^{i-1}_{G_t}(w)$.
From  $i\le r$ and  $E(G_t)\subset R_{t-1}$
we get that $a\in B^{r-1}_{R_{t-1}}(w)$,
which together with $w\in B^{r}_{G_{t-1}}(v)\subset B^{r}_{R_{t-1}}(v)$ implies $a\in B^{2r-1}_{R_{t-1}}(v)$. In particular, $a\in \bigcup \cal N$.
By definition of $E'$, this implies $ua\in E(G_t)$. Together with 
$a\in B^{i-1}_{G_t}(w)$, this implies that $u\in B^{i}_{G_t}(w)$, as required.
\end{claimproof}
\Cref{cl:k-flip} and \Cref{cl:balls} together prove \Cref{lem:strategy}.
\end{proof}

\subsection{Almost bounded merge-width}\label{sec:abmw}
A graph class $\CC$ has \emph{almost bounded merge-width}
if for every $r\in\N$ and $\eps>0$,
we have that $\mw_r(G)\le O_{\CC,r,\eps}(|V(G)|^\eps)$, for all $n\in\N$ and all $n$-vertex graphs $G\in\CC$.

Clearly, classes of bounded merge-width have almost bounded merge-width.
The following result states that all nowhere dense classes have almost bounded merge-width. We will prove this later below, by inspecting the proof of \Cref{thm:be}.

\thmnwd*

Classes of \emph{almost bounded flip-width} are defined analogously as classes of almost bounded merge-width, with $\mw_r(G)$ replaced with $\fw_r(G)$.
The following is the main result of this section.

\thmabmw*

As every hereditary class of almost bounded flip-width is monadically dependent \cite[Thm. 2.12]{flip-breakability}, we obtain the following.
\begin{corollary}\label{cor:abmw-mNIP}
    Every hereditary class of almost bounded merge-width is monadically dependent.
\end{corollary}

\Cref{thm:nwd} and \Cref{cor:abmw-mNIP} together justify our \Cref{conj:almostboundedmergewidth}
that almost bounded merge-width coincides with monadic dependence on hereditary classes.






\subsection{Preliminaries}
Before proving \Cref{thm:nwd} and \Cref{thm:abmw}, 
we first recall some basic notions.

\medskip
A \emph{set system} is a pair $(X,\cal F)$ with $\cal F\subset 2^X$.
Its \emph{VC-dimension} is the maximal size of a subset $Y\subset X$ such that 
$\setof{Y\cap F}{F\in\cal F}=2^Y$.
We recall the fundamental Sauer-Shelah-Perles lemma~\cite{sauer,shelah-sauer-lemma}.

\begin{lemma}[Sauer-Shelah-Perles lemma]\label{lem:sauer-shelah-perles}
  Let $(X,\cal F)$ be a set system of VC-dimension~$d$.
  Then $|\cal F|\le O(|X|^d)$.
\end{lemma}

The \emph{VC-dimension} of a graph $G$, denoted $\VCdim(G)$, 
is defined as the VC-dimension of the set system $\bigl(V(G),\setof{N(v)}{v\in V(G)}\bigr)$.
More explicitly, $\VCdim(G)$ is the maximal size of a subset $X\subset V(G)$
such that $\setof{N(v)\cap X}{v\in V(G)}=2^X$.

Define the \emph{neighborhood complexity} of a graph $G$, as the function $\pi_G\from\N\to\N$ defined as
 $$\pi_G(s)\coloneqq\max\Bigl\{\bigl|\{N_G(v)\cap S\,:\,v\in V(G)\}\bigr|\,:\, S\subset V(G), |S|\le s\Bigr\} \quad\text{for all }s \in \N.$$
In particular, $\pi_G(s)\le 2^s$, for all graphs $G$ and all $s\in\N$.
From \Cref{lem:sauer-shelah-perles}, we get the following.
\begin{corollary}\label{cor:sauer-shelah}
    Let $G$ be a graph of VC-dimension $d$.
    Then $\pi_G(s)\le O(s^d)$, for all $s\in \N$.
\end{corollary}

\subsection{Proof of \Cref{thm:nwd}}

To prove \Cref{thm:nwd}, we use the following result.
\begin{fact}[\cite{zhu2009colouring,sparsity-book}]\label{fact:nd-wcol}
    Let $\CC$ be a nowhere dense graph class, and fix $r\in\N$ and $\eps>0$.
    Then for every graph $G\in\CC$
    $$\wcol_r(G)\le O_{\CC,r,\eps}(|V(G)|^{\eps}).$$
\end{fact}

\begin{lemma}\label{lem:nd-VC}
    Let $\CC$ be a nowhere dense graph class. 
    Then there is some $d\in \N$ such that all graphs $G\in \CC$ have VC-dimension less than $d$.
\end{lemma}
\begin{proof}
    As $\CC$ is nowhere dense, there is some $d\in\N$ such that 
    no graph $G\in \CC$ contains the $1$-subdivision of the clique $K_d$, as a subgraph.
    It follows that every graph $G\in \CC$ has VC-dimension less than $d$.
\end{proof}

\begin{lemma}\label{lem:wcol-poly}
    For every $r\in\N\cup\set{\infty}$ and graph $G$, we have
    $$\mw_r(G)\le 3\cdot\pi_G(\wcol_{r+1}(G)).$$
  \end{lemma}
  \begin{proof}[Proof sketch]
    We follow the proof of \Cref{lem:wcol}, and use the same notation. 
    Let \(k\coloneqq\wcol_{r+1}(G)\).
    Fix $t\in [n]$, and let $L_t,S_t,\cal P_t,R_t,\le$ be as defined in the proof.
    Fix $v\in L_t$.
    It was observed that for every $w\in V(G)$ which is reachable from $v$ by a path of length $r$
in the graph $(V,R_t)$, 
 the atomic type of $w$ over $S_t$ is uniquely 
determined by $N_G(w)\cap S_t \subseteq \wreach_{r+1}(G,\le,v)$,
and moreover, $|\wreach_{r+1}(G,\le,v)|\le k$.
Since $$\Big|\setof{N_G(w)\cap \wreach_{r+1}(G,\le,v)}{w\in V(G)}\Big|\le  \pi_G(k),$$
we conclude that are at most $\pi_G(k)$ parts of $\cal P_t$ 
that are reachable by a path of length at most $r+1$ from $v$.
Thus, the radius-$r$ width of $(\cal P_t,R_t)$ is at most
$\pi_G(k)$. The rest of the argument remains unchanged, yielding the final bound $\mw_r(G)\le 3\cdot \pi_G(k)\le 3\cdot \pi_G(\wcol_{r+1}(G))$.
  \end{proof}
  
  \Cref{thm:nwd} now follows by combining the previous insights.
  \begin{proof}[Proof of \Cref{thm:nwd}]
    By \Cref{lem:nd-VC}, there is some $d\in\N$ such that every $G\in\CC$ has VC-dimension at most $d$.
    Fix $r\in\N$ and $\eps>0$.
Then, for every graph $G\in\CC$, by \Cref{lem:wcol-poly} and \Cref{fact:nd-wcol}, we have 
$$\mw_r(G)\le 3\cdot \pi_G\bigl(\wcol_{r+1}(G)\bigr)\le  3\cdot \pi_G\bigl(O_{\CC,r,\eps}(|V(G)|^\eps)\bigr)\le O_{\CC,r,\eps}\bigl(|V(G)|^{\eps\cdot d}\bigr).$$
Since $\eps>0$ is arbitrary, this implies that $\CC$ has almost bounded merge-width.
  \end{proof}

  \subsection{Proof of \Cref{thm:abmw}}

We start by proving that hereditary classes of almost bounded merge-width have bounded VC-dimension. 
In fact, we prove a stronger result, expressed using the following notion.

\newcommand{\ntn}{\mathrm{ntn}}

A pair $u,v$ of distinct vertices in a graph $G$ is a pair of \emph{$k$-near-twins} if $|N_G(u)\triangle N_G(v)|\le k$.
 For a bipartite graph $G=(X,Y,E)$, let the near-twin number $\ntn(G)$ denote 
 the smallest number $k$ such that 
for all $X'\subset X,Y'\subset Y$ with $|X'|+|Y'|>2$,
the bipartite subgraph $G[X',Y']$ induced by $X'$ and $Y'$  in $G$
has a pair of $k$-near-twins contained either in $X'$, or in $Y'$.

\begin{lemma}\label{lem:twins-bip}
    Let $G=(X,Y,E)$ be a bipartite graph with $\mw_1(G)\le k$
    and with $|X|+|Y|>2$.
    Then $G$ contains a pair of $2k$-near-twins
     contained in a single part of $G$. In particular, $\ntn(G)\le 2k$.
\end{lemma}
\begin{proof}
    Fix a construction sequence of $G$ of radius-$1$ width $k$. Consider the first moment when some part $A$ of the current partition $\cal P$ contains two vertices $a,b$ that belong to the same part of the bipartition $\set{X,Y}$ of $G$. Suppose $a,b\in X$. Then $N_G(a)\triangle N_G(b)\subset (N_R(a)\cup N_R(b))\cap Y$,
    where $N_R(\cdot)$ denotes the neighborhood in the graph $(V,R)$  of currently resolved pairs.
    Note that $|N_R(a)\cap Y|\le k$,
    since no two vertices of $Y$ are in a single part of the current partition $\cal P$, and $N_R(a)$ is contained in at most $k$ parts of $\cal P$. Similarly, $|N_R(b)\cap Y|\le k$.
    It follows that $|N_G(a)\triangle N_G(b)|\le 2k$.
\end{proof}

The following lemma is implicit in \cite{flip-width}
(it follows from 
\cite[Lem. 5.25]{flip-width-arxiv}).
\begin{lemma}\label{lem:ntn}
    If $\CC$ is a hereditary class of bipartite graphs such that $\ntn(G)\le o(|V(G)|)$ for $G \in \CC$, then $\VCdim(\CC)<\infty$.
\end{lemma}



\begin{corollary}\label{lem:abmw-vc}
    If $\CC$ is a hereditary class of graphs such that $\mw_1(G)\le o(|V(G)|)$ for $G \in \CC$, then $\VCdim(\CC)<\infty$.
\end{corollary}
    \begin{proof}
For a graph $G$ define the bipartite graph $B(G)$, with two parts of size $V(G)$,
representing the binary relation $E(G)\subset V(G)\times V(G)$.
Then $\mw_1(B(G))\le O (\mw_1(G))$, as a construction sequence of $G$ can
be easily converted to a construction sequence for $B(G)$.
In particular, $\mw_1(B(G))\le o(|V(G)|)$ for $G\in \CC$.
Moreover, $\VCdim(G)=\VCdim(B(G))$.
The conclusion follows from \Cref{lem:twins-bip}
and \Cref{lem:ntn}, applied to the class $\setof{B(G)}{G\in\CC}$.
    \end{proof}

    Recall that by the Sauer-Shelah-Perles Lemma (see \Cref{lem:sauer-shelah-perles}), 
    if $G$ is a graph of VC-dimension at most $d$,
    then $\pi_G(s)\le O(s^d)$ for all $s\in\N$. The following variant of \Cref{lem:fw}
     therefore gives -- for graphs of fixed VC-dimension -- a polynomial bound on the flip-width parameters, in terms of the merge-width parameters.

    \begin{lemma}\label{lem:fw-poly}Fix $r\ge 1$.
        For every graph $G$, if 
        $s=\mw_{2r-1}(G)$ then        
        $$\fw_r(G)\le s^2+\pi_G(s).$$
    \end{lemma}
\Cref{lem:fw-poly} strengthens 
 \Cref{lem:fw} by replacing the upper bound $4^s$ by $s^2+\pi_G(s)$.
    Then \Cref{lem:fw-poly} follows from the lemma below, 
    exactly in the same way as \Cref{lem:fw} follows from \Cref{lem:strategy}.

\begin{lemma}\label{lem:strategy-poly}
Fix \(r\ge 1\) and \(t\in[2,n]\), and let \(s\) be the radius-\((2r-1)\)
width of \((\cal P_t,R_{t-1})\). For every \(v\in V(G)\), there is an
\((s^2+\pi_G(s))\)-flip \(G_t'\) of \(G\) such that
\[
   B^r_{G_t'}(w)\subseteq B^r_{G_t}(w)
   \qquad\text{for every }w\in B^r_{G_{t-1}}(v).
\]
\end{lemma}

\begin{proof}
    Denote $V\coloneqq V(G)$ and 
fix \(v\in V\). Let $\cal N$ 
be the set of parts of \(\cal P_t\) which contain a vertex at distance at most
\(2r-1\) from \(v\) in the graph \((V,R_{t-1})\). 
By the assumption on the radius-\((2r-1)\) width of
\((\cal P_t,R_{t-1})\), we have
\[
   |\cal N|\le s.
\]

Define a graph \(G'=(V,E')\) as follows. For every two parts
\(A,B\in\cal P_t\), set
\[
E'\cap AB =
\begin{cases}
E(G)\cap AB, & \text{if } A,B\notin \cal N,\\
E(G_t)\cap AB, & \text{otherwise.}
\end{cases}
\tag{1}\label{eq:poly-Gprime-def}
\]
Note that $G'$ is exactly the same graph as constructed in the proof of \Cref{lem:strategy}. In particular, \Cref{cl:balls} implies that 
$$B^r_{G'}(w)\subset B^r_{G_t}(w)\quad\text{for every $w\in B^r_{G_{t-1}}(v)$.}$$
It therefore remains to show that \(G'\) is an \((s^2+\pi_G(s))\)-flip of \(G\).

For every \(A\in\cal N\), choose a vertex
\[
   a_A\in A\cap B^{2r-1}_{(V,R_{t-1})}(v),
\]
and put
\[
   S\coloneqq \{a_A:A\in\cal N\}.
\]
Then \(|S|=|\cal N|\le s\).

For \(A\in\cal N\), define
\[
   \mathcal X_A
   \coloneqq
   \{A\}\cup
   \{D\in\cal P_t:N_{G_t}(a_A)\cap D\neq\emptyset\}.
\]
Since \(E(G_t)\subseteq R_t\subseteq R_{t-1}\), every part in
\(\mathcal X_A\) intersects the radius-\(1\) ball of \(a_A\) in
\((V,R_{t-1})\). Since \(1\le 2r-1\), the radius-\((2r-1)\) width gives
\[
   |\mathcal X_A|\le s.
\]
Let
\[
   \mathcal X\coloneqq \bigcup_{A\in\cal N}\mathcal X_A.
\]
Then
\[
   \cal N\subseteq \mathcal X
   \qquad\text{and}\qquad
   |\mathcal X|\le |\cal N|\cdot s\le s^2.
\]

Let
\[
   U\coloneqq \bigcup \mathcal X.
\]
Let \(\mathcal R\) be the partition of \(V\setminus U\) in which two
vertices \(y,z\in V\setminus U\) are equivalent iff
\[
   N_G(y)\cap S=N_G(z)\cap S.
\]
Finally define
\[
   \mathcal Q\coloneqq \mathcal X\cup\mathcal R.
\]
Then \(\mathcal Q\) is a
partition of \(V\). Moreover,
\[
   |\mathcal Q|
   =
   |\mathcal X|+|\mathcal R|
   \le
   s^2+
   \bigl|\{N_G(y)\cap S:y\in V\}\bigr|
   \le
   s^2+\pi_G(s).
\tag{2}\label{eq:poly-Q-size}
\]

We claim that \(G'\) is a \(\mathcal Q\)-flip of \(G\). It is enough to show
that for every \(A,B\in\mathcal Q\),
\[
   E'\cap AB=E(G)\cap AB
   \qquad\text{or}\qquad
   E'\cap AB=AB\setminus E(G).
\tag{3}\label{eq:poly-flip-check}
\]

First suppose \(A,B\in\mathcal X\). Then \(A\) and \(B\) are parts of
\(\cal P_t\). By definition of \(G'\), the graph \(G'\) agrees on \(AB\) either
with \(G\) or with \(G_t\). Since \(G_t\) is a \(\cal P_t\)-flip of \(G\),
\eqref{eq:poly-flip-check} follows.

Next suppose \(A,B\in\mathcal R\). Every \(\cal P_t\)-part meeting \(A\cup B\)
lies outside \(\mathcal X\), hence outside \(\cal N\). Therefore
\eqref{eq:poly-Gprime-def} gives
\[
   E'\cap AB=E(G)\cap AB.
\]

Now suppose \(A\in\mathcal X\setminus\cal N\) and \(B\in\mathcal R\).
Again, every \(\cal P_t\)-part meeting \(B\) lies outside \(\mathcal X\), hence
outside \(\cal N\), and \(A\notin\cal N\). Therefore
\[
   E'\cap AB=E(G)\cap AB.
\]

It remains to consider the case \(A\in\cal N\) and \(B\in\mathcal R\).
Let \(D\in\cal P_t\) be a part with \(D\cap B\neq\emptyset\). Since
\(B\cap U=\emptyset\), we have \(D\notin\mathcal X\). In particular,
\(D\notin\mathcal X_A\), and therefore
\[
   N_{G_t}(a_A)\cap D=\emptyset.
\tag{4}\label{eq:no-Gt-neighbor}
\]
Since \(G_t\) is a \(\cal P_t\)-flip of \(G\), on the pair \(AD\) either
\(G_t\) agrees with \(G\), or \(G_t\) agrees with the complement $\overline G$ of \(G\). By
\eqref{eq:no-Gt-neighbor}, these two alternatives are distinguished by the
\(G\)-adjacency of vertices of \(D\) to \(a_A\):
\[
\begin{array}{ll}
G_t \text{ agrees with } G \text{ on } AD
   &\Longleftrightarrow
   a_A \text{ is anti-complete to } D \text{ in } G,\\[1mm]
G_t \text{ agrees with } \overline G \text{ on } AD
   &\Longleftrightarrow
   a_A \text{ is complete to } D \text{ in } G.
\end{array}
\]
But \(B\) is a part of \(\mathcal R\), so all vertices of \(B\) have the same
\(G\)-adjacency to \(a_A\), because \(a_A\in S\). Hence the same alternative,
agreement with \(G\) or agreement with \(\overline G\), occurs for every
\(\cal P_t\)-part \(D\) meeting \(B\). Since \(A\in\cal N\), we have
\[
   E'\cap AD=E(G_t)\cap AD
\]
for every such \(D\). Therefore \(E'\cap AB\) is either \(E(G)\cap AB\) or
\(AB\setminus E(G)\). This proves \eqref{eq:poly-flip-check}.

Thus \(G'\) is a \(\mathcal Q\)-flip of \(G\). By
\eqref{eq:poly-Q-size}, \(G'\) is an \((s^2+\pi_G(s))\)-flip of \(G\). Taking
\(G_t'\coloneqq G'\) completes the proof.
\end{proof}

   As mentioned, \Cref{lem:strategy-poly} implies \Cref{lem:fw-poly},
   analogously as in the proof of \Cref{lem:fw}.

\thmabmw*
\begin{proof}
    Let $\CC$ be a hereditary class of almost bounded merge-width.
    By \Cref{lem:abmw-vc}, $\CC$  has VC-dimension bounded by some $d\in\N$.
By \Cref{lem:sauer-shelah-perles}, we have that $\pi_G(s)\le O(s^d)$ for all $G\in\CC$ and $s\in\N$.

Fix $r\in\N$ and $\eps>0$. Let $\delta>0$ be sufficiently small.
Then, for every graph $G\in\CC$ with $n$ vertices and  $s=\mw_{2r-1}(G)$ 
we have $s\le O_{\CC,r,\delta}(n^\delta)$ and $\pi_G(s)\le O(s^d)$, so by \Cref{lem:fw-poly}:
$$\fw_r(G)\le s^2+\pi_G(s)\le  O(s^{d+2})\le  O_{\CC,r,\delta}(n^{(d+2)\delta}).$$
Setting $\delta=\eps/(d+2)$ gives $\fw_r(G)\le O_{\CC,r,\eps}(n^\eps)$. Thus $\CC$ has almost bounded flip-width.
\end{proof}